\documentclass[11,letterpaper,reqno]{article}
\usepackage[letterpaper, hmargin=2cm, vmargin={2cm, 2cm}]{geometry}
\usepackage{myDefs}

\def\D{\ensuremath{\mathbb D} }
\def\R{\ensuremath{\mathbb R} }
\def\C{\ensuremath{\mathbb C} }

\def\N{\ensuremath{\mathbb N} }

\newcommand{\diam}{\text{diam}}

\newcommand{\supp}{\operatorname{supp}}

\def\A{\ensuremath{\mathbb A} }

\numberwithin{equation}{section}

\newcommand*\M[1]{\textcolor{black}{#1}} 
	\definecolor{applegreen}{rgb}{0.55, 0.71, 0.0}

\begin{document}
\title{Simply Connected Wandering Domains of Small Order Entire Functions}
\author{Adi Gl\"ucksam and Leticia Pardo-Sim\'{o}n }
\date{}
\maketitle

\begin{abstract}
We show that any bounded, simply connected domain with analytic boundary can be realised as a wandering domain of an entire function of any prescribed order in $(0, 1)$. Extending results of Boc Thaler, our construction simultaneously prescribes the domain and the exact order of the map. In particular, we produce the first examples of entire functions with bounded simply connected wandering domains of each order in $\bigl(0,\tfrac12\bigr)$.
\end{abstract}

\section{Introduction}\label{sec:intro}
Let $f\colon\C\to\C$ be an entire function and denote by $f^n$ its $n$-th iterate, $n\in \N$. The \emph{Fatou set}, $F(f)$, consists of the points $z\in\C$ for which the family $\{f^n\}_{n\in\N}$ is normal in a neighbourhood of $z$, and the \emph{Julia set} is its complement, $J(f):=\C\setminus F(f)$. We refer to the surveys \cite{Bergweiler_1993,Schleicher_entire,Sixsmith_2018,Bergweiler_Rempe_escaping_survey_25} for background on iteration of entire functions. A connected component of the Fatou set, $U\subset F(f)$, is called a \emph{wandering domain} if $f^n(U)\cap f^m(U)=\emptyset$ whenever $n\neq m$. By Sullivan’s no-wandering-domain theorem, \cite{Sullivan_noWD_85}, rational maps, and, in particular, polynomials, admit no wandering domains, i.e., this phenomenon is inherently transcendental.

The first example {of an entire map with a} wandering domain was constructed by Baker in \cite{Baker_wandering_76}. Since then, many examples have been produced by various methods. Notably, constructions based on the Eremenko–Lyubich approximation scheme \cite{EL_87} —rooted in classical approximation theorems of Runge and Arakelyan—offer significant flexibility. This flexibility allows one to prescribe features such as the family of iterates having finite limit functions (yielding \textit{oscillating} wandering domains) or even prescribe the exact shape of the domains, e.g. \cite{BocThaler_2021, MartiPete_JAMS_2025}. The trade-off of these methods, however, is that typically they provide little quantitative information about the map itself—most prominently about its set of singular values (singularities of the inverse) or its  \emph{order of growth}, defined by
\[
\rho(f)=\limsup_{r\to\infty}\frac{\log\log M_f(r)}{\log r}, \qquad 	M_f(r):=\max_{\abs z=r}\abs{f(z)}.
\]

By contrast, constructions based on quasiconformal techniques (e.g. \cite{Kisaka_08,Bishop_acta_2015,Marti_Pete_2020,Burkart_25}), \textit{lifts} of periodic components (e.g. \cite{Herman_84,Fagella_henriksen_Teich_09}), or infinite products (e.g. \cite{Baker_85_mc_finiteorder,Bishop_dim1_18,Sixsmith_2018}) can produce entire functions with wandering domains and bounded singular sets and/or finite order; however, they typically offer limited control over the geometry of the domains themselves.

In keeping with a central theme of transcendental dynamics—identifying minimal transcendental features, such as growth, that still permit non-rational behaviour (here, wandering domains)—a natural goal is to reduce the order of growth as much as possible without losing geometric control. In this work we achieve simultaneous control on both the growth-rate and the geometry: we prescribe both the analytic boundary of any bounded simply connected wandering domain and the (arbitrarily small) order of growth of the corresponding entire function.

\begin{thm}\label{thm:intro}
	For every $\alpha\in\bb{0,1}$ and every bounded simply connected domain, $U$, with analytic boundary there exists an entire function, $f$, of order $\alpha$ for which $U$ is a wandering domain and the iterates, $f^n\vert_{U}$, are univalent.
\end{thm}

We now place the result within its broader context. Wandering domains for entire maps may be simply connected or multiply connected. The multiply connected case is relatively well understood: multiply connected wandering domains are necessarily bounded, their iterates form a nested family of annuli tending to infinity and the dynamics are power-map-like, i.e., broadly speaking, they can really only posses one type of geometry and one type of dynamics (see, e.g., \cite{Bergweiler_PG_multiply_13} for more details). Furthermore, building on his original example, Baker showed in 1985 that for every $\rho\in[0,\infty]$ there exists an entire function of order $\rho$ with a multiply connected wandering domain \cite{Baker_85_mc_finiteorder}. Bishop later gave, via an infinite product, an example of an entire map of order $0$ with multiply connected wandering domains whose boundaries form a countable collection of $C^1$ curves, being this the first example of an entire function with Julia set of Hausdorff dimension $1$, \cite{Bishop_dim1_18}.

By contrast, simply connected wandering domains display a much wider range of phenomena. For example, they could be bounded or unbounded (\M{see}, e.g., \cite{Herman_84,Lazebnik_2021}), the set of limit functions of $\bset{f^n\vert_{U}}$ may include finite constant ones, i.e., they can be \emph{oscillating} (first example in \cite{EL_87}), and they arise for functions in class~$\mathcal{B}$ (entire functions with \M{a} bounded singular set), see \cite{Bishop_acta_2015} followed by \cite{Lazebnik_2017,Fagella_2019,Marti_Pete_2020}. 
Moreover, their internal dynamics admit several distinct types \M{as seen in} \cite{Benini_2022,Evdoridou_2023,Evdoridou_fast_2023}; see also \cite{Gustavo} for analogous results in the meromorphic setting.

The possible geometries of simply connected wandering domains were recently unravelled by Boc Thaler~\cite{BocThaler_2021}. They constructed for any bounded, simply connected, regular open set, $U$, an entire function realising $U$ as an escaping (or oscillating) univalent wandering domain. Building on these ideas, Martí-Pete, Rempe and Waterman \cite{MartiPete_JAMS_2025} further broadened the class of admissible shapes by realising escaping wandering domains for any bounded simply connected domain $D$ such that $\C\setminus \overline{D}$ has an unbounded connected component $W$ with $\partial W = \partial D$, thus supporting highly intricate geometries such as Lakes of Wada continua; see also \cite{MartiPete_merom_2025} for meromorphic analogues. However, none of these constructions provides any bound on the growth-rate of the entire map.

Theorem \ref{thm:intro} requires the analyticity of $\partial U$. \M{Such sets are bounded, simply connected, regular open sets}, i.e., they fall within  Boc Thaler's setting in \cite{BocThaler_2021} because of the conformal extendability of the Riemann map across~$\partial U$ (see \cite[Prop.~3.1]{pommerenke_boundary}). In particular, \cite{BocThaler_2021} shows that boundaries of bounded simply connected wandering domains can be made arbitrarily well-behaved, in contrast to invariant Fatou components, whose boundaries are typically highly non-smooth, see \cite{Azarina_89,Baranski_2025}. Theorem~\ref{thm:intro} demonstrates that this contrast in boundary regularity between wandering domains and invariant Fatou components persists even for functions of very slow growth.

Looking for the lowest possible growth-rate of entire maps admitting simply connected wandering domains, Sixsmith \cite{dave_fast12} constructed an entire function of order zero, defined via an infinite product, that possesses a simply connected wandering domain.  For higher growth, Baker \cite[Theorem~5.2]{Baker_wandering84} produced simply connected wandering domains for every order $\rho\ge1$, and Martí-Pete–Shishikura constructed functions in class $\mathcal B$ of orders $\frac p2$ for all $p\in \N$, in particular of order $\frac12$—the smallest possible within class $\mathcal B$. To the best of our knowledge, our result gives the first examples with \emph{bounded} simply connected wandering domains for all orders strictly between $0$ and $1$, thereby filling the gaps between the order-$0$ construction of Sixsmith, the order-$\frac12$ class-$\mathcal B$ examples of Martí-Pete and Shishikura, and Baker’s examples for $\rho\ge1$.

\begin{cor}For every $\rho\in [0,\infty]$ there exists and entire function of order $\rho$ with a bounded simply connected wandering domain. 
\end{cor}
This is the simply connected analogue of Baker’s theorem for multiply connected wandering domains \cite{Baker_85_mc_finiteorder}.

We remark that, in the \emph{unbounded} case, Baker’s 1981 conjecture \M{in} \cite{Baker_conjecture_81} asserts \M{the non-existence} of unbounded wandering domains for functions of order $<\tfrac12$; this remains open. Examples of order $>\tfrac12$ appeared in \cite{Evdoridou_fast_2023}; see also \cite{Nicks_PG_Baker18} for the case of real functions with real zeros.

The proof of Theorem \ref{thm:intro}, built on Hörmander’s $\bar\partial$-techniques, combines the approximation paradigm of \cite{BocThaler_2021,MartiPete_JAMS_2025}, with the approximation tools first used in \cite{Evdoridou_fast_2023} that allowed adding the quantitative control needed \M{to prescribe} small order. In the spirit of Boc Thaler, we construct a sequence of entire maps converging locally uniformly and use designated sets of points, distributed in `layers' accumulating at $\partial U$, whose forward images are mapped into attracting basins in order to shape the wandering domain. The two key new ingredients are \M{using a method that allows us to simultaneously control the approximation} error and the \emph{growth}, and a \emph{geometric} quantitative `straightening-to-a-disk' step near almost round analytic boundaries. A carefully chosen parameter scheme and a diagonal convergence argument then yield the desired map with prescribed order and wandering dynamics.

In order to obtain the prescribed set as a wandering domain, we need to precisely preserve the distinction between images of the set and its complement under iterations. We must therefore use conformal maps to map our wandering domain forward. On the other hand, there is no conformal map that changes the modulus of annuli. This implies that the distance between iterations of different `layers' surrounding the set, which are mapped to an attracting basin, and iterations of the set, which are translated forward, decreases, which naturally increases the growth-rate. It requires meticulous analysis to balance between the growth-rate bounds and the error accumulated along different iterates, resulting in a highly technical and {long} proof.

\noindent\textbf{Structure of the article.}
Section 2 sets notation and recalls the $\bar{\partial}$–approximation framework, collecting the subharmonic tools (gluing and puncture lemmas) and quantitative Riemann–mapping estimates used later. Section~3 proves Lemma 3.1, a `straightening and separating' lemma, which, for a given bounded simply connected domain with boundary that is nearly round, yields an entire function that uniformly `straightens' the domain with derivative control, sends selected small disks near the boundary into a tiny neighbourhood of a fixed far point, and satisfies an explicit global growth bound. Section 4 assembles the sequence: we fix the parameter scheme, carry out the base `straightening' step, perform the inductive step, control the attracting basin, and derive the growth estimates. Section 5 passes to the limit and completes the proof of Theorem \ref{thm:intro}.
\section{Preliminaries}\label{sec:pre}

\subsection{Notation}
\begin{itemize}
	\item We write $f^n$ for the $n$-th iterate of a function $f$. For $x>0$ and $n\in\N$, $\log^n x$ always means $\log$ to the power~$n$, and not iterated logarithm.
	\item[$\bullet$] Given a domain $\Omega\subset\C$ and $\varepsilon>0$, let $\Omega^{+\varepsilon}:=\{z\in\C:\operatorname{dist}(z,\Omega)<\varepsilon\}$, where $\operatorname{dist}(\cdot,\cdot)$ denotes the Euclidean distance.
	\item[$\bullet$]  We write $dm$ for Lebesgue's measure on the complex plane. 
	\item [$\bullet$] For $z\in\C$ and $r>0$, let $B(z,r)$ be the open disk centred at $z$ with radius $r$; we abbreviate the unit disk by~$\D$.
	\item[$\bullet$] Given a set $T\subset\C$ and $a\in\C$, \M{the translation of $T$ by $a$ is denoted} by $T+a:=\{z+a:z\in T\}$, and its scaling by $a$ \M{is denoted }by $aT:=\{a\,z:z\in T\}$.
	\item[$\bullet$] For $0<r<R$, the annulus with inner radius $r$ and outer radius $R$ is $\mathbb{A}(r,R):=\{z\in\C:r<|z|<R\}$.
	\item[$\bullet$] By a \textit{numerical constant} we mean a constant that does not depend on any parameter.
\end{itemize}

\subsection{The approximation method}\label{subsec:method}
We shall construct our desired entire functions \M{using} an approximation argument based on H\"ormander's $\bar\partial$-theorem. We use the following one-dimensional version \M{of} \cite[Theorem~4.2.1]{Hormander}.  
\begin{thm}[H\"ormander]\label{thm:Hormander}
	Let $u:\C\rightarrow\R$ be a subharmonic function. Then, for every $g\in L^2_{\text{loc}}(\C)$ there exists $\beta \in L^2_{\text{loc}}(\C)$ solving $\bar\partial \beta=g$ in the sense of distributions such that
	\begin{equation}\label{integral_Hormander}
		\iint_\C\abs {\beta(z)}^2\frac{e^{-u(z)}}{\bb{1+\abs z^2}^{2}}dm(z)\le\frac12 \iint_\C \abs {g(z)}^2e^{-u(z)}dm(z)=: {\mathcal{I}^2},
	\end{equation}
	provided that ${\mathcal{I}^2}$ is finite.
\end{thm}

We use this theorem to approximate a model map that prescribes different dynamics in different regions—when done correctly, the approximation inherits the same dynamical behaviour.
\begin{prop}\label{prop:chi}
	Let $\{\Omega_k\}$ be a collection of pairwise $4\varepsilon$–separated planar domains, i.e., $\dist(\Omega_j,\Omega_k)\ge4\eps$ for all $j\neq k$. There exists a numerical constant, $C>0$, and a smooth function, $\chi:\C\to[0,1]$, such that:
	\begin{enumerate}
		\item $\chi\equiv1$ on each $\Omega_k$.
		\item $\chi\equiv 0$ on $\C\setminus \bigcup_k \Omega_k^{+\varepsilon}$.
		\item $\supp\nabla\chi\subset \bigcup_k \big(\Omega_k^{+\varepsilon}\setminus\Omega_k\big)$ and $\|\nabla\chi\|_{L^\infty}\le \frac C\eps$.
	\end{enumerate}
\end{prop}

For details on the construction of such maps, see for example \cite[Proposition 2.5 and Observation 2.6]{GPS2024}.


The technique we now describe was previously used in various contexts (e.g., \cite{Russakovskii_93}, \cite{Adi_measurably_19,Adi_translation_19, GPS2024, Ganny2025}). In particular, we previously introduced this technique (jointly with Evdoridou in \cite{Evdoridou_fast_2023}) to construct fast escaping wandering domains.  We prescribe the desired dynamics via a \emph{model map},  $h$, that is holomorphic on the union of pairwise separated domains, \M{$\bigcup_k \Omega_k^{+\varepsilon}$,} without assuming any global holomorphic extension. We then glue these local pieces with the cut-off \M{function} $\chi$, described in Proposition~\ref{prop:chi}, and correct the resulting $\bar\partial$–error by solving the equation $\bar\partial\beta=\bar\partial(\chi h)=(\bar\partial\chi)\cdot\,h$ using Theorem~\ref{thm:Hormander} with a carefully designed subharmonic weight, $u$. The separation obtained by multiplying by $\chi$, ensures that $(\bar\partial\chi)$ is supported in thin bands around $\Omega_k$, where \M{the subharmonic function, }$u$, will be large, which makes the integral in \eqref{integral_Hormander}{, $\mathcal{I}^2$,} finite and the resulting error estimates small. The following proposition summarizes the application of the method.

\begin{prop}\label{prop:method} 
Let $\{\Omega_j\}$ be a collection of planar domains satisfying the separation hypothesis of Proposition~\ref{prop:chi}, and denote by $\chi$ the smooth function constructed in Proposition~\ref{prop:chi}. Suppose $h$ is holomorphic on $\bigcup_j \Omega_j^{+\varepsilon}$, and let $u:\C\to\R$ be a subharmonic function satisfying
$$
\mathcal{I}:=\sqrt{\frac12 \iint_\C \abs {(\bar\partial\chi)\,h}^2e^{-u(z)}dm(z)}<\infty.
$$
Let $\beta$ be Hörmander’s solution to the $\bar\partial$-equation $\bar\partial\beta=\bar\partial(\chi h)=(\bar\partial\chi)\cdot\,h$ provided by Theorem~\ref{thm:Hormander}, and define the map $f\colon \C\to\C$ 
\[
f(z):=\chi(z)\,h(z)-\beta(z).
\]
Then $f$ is entire and satisfies:
\begin{enumerate}[label=$(A_\arabic*)$]
	\item \label{item:error_bound} (Error bound) For every $z\in\C$ and every $r\in \left(0,\max\bset{1,\frac{\abs z}3}\right]$ for which $\chi|_{B(z,r)}\equiv 1$,
	$$
	\abs{f(z)-h(z)} \leq \frac{2+2\abs z^2}{{r}}\cdot {\mathcal{I}}\cdot \underset{w\in B(z,r)}\max\; e^{\frac12 u(w)}.
	$$
	\item \label{item:growth_bound}(Growth bound) For every $z\in \C$, 
	$$
	\abs{f(z)}\leq \underset{w\in B(z,1)}\max\abs{h(w)}+\frac{2+2\abs z^2}{{\sqrt \pi}}\cdot {\mathcal{I}}\cdot \underset{w\in B(z,1)}\max\; e^{\frac12 u(w)}.
	$$
\end{enumerate}
\end{prop}
\begin{proof}
By Proposition \ref{prop:chi},  $\dbar\chi$ is bounded, and contained in a region where $h$ is holomorphic, which implies that  $g\in L^2_{\mathrm{loc}}(\C)$. This, combined with the proposition assumptions, means that all hypotheses of Hörmander's theorem are satisfied. Observe that $f\in L^1_{\mathrm{loc}}(\C)$, since $\chi$ is bounded, and $\beta\in L^2_{\mathrm{loc}}(\C)\subset L^1_{\mathrm{loc}}(\C)$. Moreover, since the equation $\bar\partial \beta=g$ in the statement of Hörmander's theorem holds in distribution, we obtain that 
\begin{equation*}
	\overline\partial f(z)=\overline\partial\bb{ \chi(z) \cdot h(z)-\beta(z)}=\overline\partial \chi(z)\cdot h(z)-\overline\partial\beta(z)=g(z)-g(z)=0
\end{equation*}
in distribution. By Weyl's lemma, weak solutions for the Laplace equation are smooth solutions, see e.g. \cite[p. 67, Remark 57]{JabbariSCV2020}, therefore the equality above holds point-wise, i.e., $f$ is entire.

In order to prove property \ref{item:error_bound}, note that for every $z\in\C$ and $r>0$ for which $\chi(w)=1$ for all $w\in B(z,r)$, $\beta(w)= f(w)-h(w)$, i.e., $\beta$ is holomorphic on $B(z,r)$. Applying Cauchy’s integral formula for the function $\beta=f-h$ and using Hörmander's theorem, 
\begin{align*}
	\abs{f(z)-h(z)}&=\abs{\beta(z)}\le \frac 1{\pi r^2}\abs{\iint_{B(z,r)}\beta(w)dm(w)}\le \frac1{ {\sqrt{\pi}r}}\sqrt{\iint_{B(z,r)}\abs{\beta(w)}^2dm(w)}\nonumber\\
	&<\frac1{ {r}} \underset{w\in B(z,r)}\max\; e^{\frac12 u(w)}\bb{1+\abs w^2}\sqrt{\iint_{B(z,r)}\abs{\beta(w)}^2\frac{e^{-u(w)}}{\bb{1+\abs w^2}^2}dm(w)}\nonumber\\
	&\le \frac1{ {r}}\underset{w\in B(z,r)}\max\; e^{\frac12 u(w)}\bb{1+\abs w^2}\cdot\mathcal I\le \frac{2+2\abs z^2}{{r}}\cdot {\mathcal{I}}\cdot \underset{w\in B(z,r)}\max\; e^{\frac12 u(w)}, 
\end{align*}
where the second inequality is due to Cauchy–Schwarz inequality, and the last inequality holds since for $r\leq \frac{\abs z}3$
$$\underset{w\in B(z,r)}\max\left(1+\abs w^2\right)\leq 1+ \bb{\abs z +r}^2\leq 2\bb{1+\abs{z}^2}.$$

To see property \ref{item:growth_bound}, arguing similarly, with $r=1$ and using Minkowski inequality (\M{the triangle inequality of the $L^2$ norm}), 
\begin{align*}
	\abs{f(z)}&=\frac{1}\pi\abs{\integrate{B\bb{z,1}}{}{f(w)}{m(w)}}
	\le \frac1{\sqrt \pi}\sqrt{\integrate{B\bb{z, 1}}{}{\abs{\chi(w)\cdot h(w)-\beta(w)}^2}{m(w)}}\\
	&\le \frac1{\sqrt \pi}\underset{w\in B(z,1)}\max\abs{h(w)}+\frac1{\sqrt \pi}\sqrt{\integrate{B\bb{z,1}}{}{\abs{\beta(w)}^2}{m(w)}}\\
	&\le \underset{w\in B(z,1)}\max\abs{h(w)}+\frac{2+2\abs z^2}{{\sqrt \pi}}\cdot {\mathcal{I}}\cdot \underset{w\in B(z,1)}\max\; e^{\frac12 u(w)},
\end{align*}
concluding the proof.
\end{proof}
Note that in order to apply this method successfully, we need to carefully craft a subharmonic function that allows us to apply Proposition \ref{prop:method}. On one hand, the subharmonic function must guarantee that the integral in \eqref{integral_Hormander} is finite; on the other hand, it must ensure that the bounds in the error, \ref{item:error_bound}, and growth, \ref{item:growth_bound}, are satisfactory. 

To apply Proposition \ref{prop:method} we need to carry out the following steps:
\begin{enumerate}
\item Define the model map.
\item \M{Find a suitable} subharmonic map.
\item Define the cut-off map, $\chi$, and show the corresponding integral in \eqref{integral_Hormander} is finite.
\item Apply Proposition \ref{prop:method} to obtain an entire map with subsequent error and growth bounds.
\end{enumerate}


\subsection{Subharmonic functions}

As we saw in the previous subsection, subharmonic functions play a central role in our construction:  they make the $L^2$ integral finite and keep the error and growth bounds under control where needed. We include two of the main results we shall use, and refer the reader to \cite{Ransford, hayman_subharmonic} for basic definitions and standard results.

We require our subharmonic \M{function} to satisfy different properties in different regions. Specifically, where we require good approximation, the function needs to be small (potentially negative), but in `gluing' regions (i.e., where $\partial \chi$ is supported), the function needs to be large. We shall often use the following well-known `gluing lemma' and its corollary.

\begin{thm}[{\cite[Theorem 2.4.5]{Ransford}}]
Let $u$ be a subharmonic function on an open set $U\subset\mathbb{C}$, and let $v$ be a subharmonic function on an open subset $V$ of $U$ such that
$\limsup_{z \to \zeta} v(z) \leq u(\zeta)$ for $\zeta \in U \cap \partial V.$
Then $\tilde{u}$ is subharmonic on $U$, where
\[\tilde{u} =
\begin{cases}
	\max\bset{u,v} & \text{on } V, \\
	u & \text{on } U \setminus V.
\end{cases}
\]
\end{thm}

\begin{cor} \label{cor:gluing} Let $u$ and $v$ be subharmonic functions defined on open sets $U,V\subset \C$ respectively, such that $U\cap V\neq \emptyset$. If 
$$ \limsup_{z \to \zeta} v(z) \leq u(\zeta) \text{ for } \zeta \in U \cap \partial V \quad \text{ and }	 \quad \limsup_{z \to \zeta} u(z) \leq v(\zeta) \text{ for } \zeta \in V \cap \partial U,$$
then $w$ is subharmonic on $U\cup V$, where
\[w =
\begin{cases}
	u & \text{on } U \setminus V,\\
	\max\bset{u,v} & \text{on } U\cap V, \\
	v & \text{on } V \setminus U.
\end{cases}
\]		
\end{cor}
The following result allows \M{one} to modify a given subharmonic function by \M{essentially} assigning the value $\bb{-\infty}$ to prescribed points, which \M{gives rise to} a new subharmonic function taking very negative values in neighbourhoods of such points, see \cite[Figure 1]{Ganny2025}. \M{Such modification is} particularly useful when such points are in the regions where we want our entire map to approximate the model map \M{closely, for example} part \ref{item:error_bound} of Proposition~\ref{prop:method} \M{could become} arbitrarily small around them.
\begin{lem}[The Puncture Lemma {\cite[Lemma 2.1]{Ganny2025}}]\label{lem:punctures}
Let $u\colon \C \to \R$ be a continuous subharmonic function and, for $N\in\N \cup \{\infty\}$, let $\bset{B_k=B(z_k,r_k)}_{k=1}^N$ be a collection of pairwise disjoint disks. Assume that for every $k$, $\underset{z\in B_k}\inf\; \Delta u >0$. Then there exists a subharmonic function $v\colon \C \to \R$ satisfying
\begin{enumerate}
	\item $u(z)=v(z)$ whenever $z\nin\bunion k 1 N B_k$;
	\item \label{item:punctures} For every $\delta\in\bb{0,1}$, and $k$,
	$$
	\underset{z\in B(z_k,\delta\cdot r_k)}\max\;v(z)\le \underset{z\in B_k}\max\; u(z)-c\cdot r_k^2\cdot \left(\underset{z\in B_k}\inf\; \Delta u\right)\cdot \log\bb{\frac1\delta},
	$$
	where $c>0$ is some numerical constant.
\end{enumerate}
\end{lem}


\subsection{Conformal maps}
\M{One of the key ingredients in our proof is repeatedly straightening the iterates of $U$ back to the unit disk. We do so by using a Riemann map while keeping quantitative control of how close the straightening is to the identity when the these iterates are nearly a disk.  }The following result of Warschawski becomes useful for this goal.

\begin{lem}[Warschawski, {\cite[Lemma 3]{War50}}]\label{lem:war}
Let $G\subset\C$ be a simply connected domain satisfying that
\begin{equation}\label{eq:war}
	\partial G\subset \A\left( \frac1{1+\eps}, 1+\eps\right)  \quad \text{ for some } \eps>0.
\end{equation}
Let $h:\D\rightarrow G$ be the Riemann map normalised so that $h(0)=0$ and $\arg(h'(0))>0$. Then for every $r<1$
$$
\underset{\abs z\leq   r}\max\;\abs{h(z)-z}\le r\eps e^\eps\bb{1+\frac{2}{\pi}\log\bb{\frac{1+r}{1-r}}}.
$$
\end{lem}

Note that Lemma \ref{lem:war} deals with bounds on maps from the disk to a domain. Our construction requires an inverse form, given by the following lemma.
\begin{lem}\label{lem:riemann_approx} Let $G\subset\C$ be a simply connected domain satisfying \eqref{eq:war} for some $\eps\in\bb{0,e^{-2}}$. Then 
$$
\abs z \;<\; 1 - 3\,\varepsilon\log\bb{\frac{1}{\varepsilon}}\Rightarrow  \abs{h^{-1}(z)-z} \;<\; 3\,\varepsilon\log\bb{\frac{1}{\varepsilon}}.
$$
\end{lem}

\begin{proof}

The idea of the proof is to understand what is the image of $R\cdot \D$ under $h\inv$, and use Warschawski's lemma, Lemma~\ref{lem:war}, on the new radius, i.e., if $h\inv(R\cdot\D)\subset B(0,\rho(R))\subseteq\D$ \M{then}
$$
\sup_{|z|<R}|h^{-1}(z)-z|  \le \max_{|w|\le \rho(R)}|w-h(w)|\le \eps e^\eps\bb{1+\frac{2}{\pi}\log\bb{\frac{1+\rho(R)}{1-\rho(R)}}},
$$
\M{by inclusion of the domains and monotonicity of function on the right.}

We will first bound $\eps e^\eps\bb{1+\frac{2}{\pi}\log\bb{\frac{1+r}{1-r}}}$ for suitable $\eps$ and $1-r<e^{-2}$: since \M{ whenever $r>\frac34$ we have }\(\log\left(1+r\right)\le \frac12\log\left(\frac{1}{1-r}\right)\), 
\[
1+\frac{2}{\pi}\log\left(\frac{1+r}{1-r}\right)
=1+\frac{2}{\pi}\left(\log\left(1+r\right)+\log\left(\tfrac{1}{1-r}\right)\right)\le 1+\frac3\pi\log\bb{\frac1{1-r}}\le \frac32\cdot \log\left(\frac{1}{1-r}\right),
\]
implying, together with Warschawski's lemma, Lemma~\ref{lem:war}, that for every $1-e^{-2}\le \abs w<1$, and $\eps<\frac14$
\begin{equation}\label{eq:Warsch}
	\begin{aligned}
		|w-h(w)|\le \abs w\eps e^\eps\bb{1+\frac{2}{\pi}\log\bb{\frac{1+\abs w}{1-\abs w}}}\le 2\eps \cdot\log\bb{\frac{1}{1-\abs w}}.
	\end{aligned}	
\end{equation}
Moreover, since $h$ is conformal, for every $r$, $\partial h\bigl(B(0,r)\bigr)=h\bigl(\partial B(0,r)\bigr)$, i.e., for every $w$ with $1-\abs w\in\bb{0,\M{e^{-2}}}$,
$$
\abs{h(w)}	\begin{cases}
	\ge \abs w-2\eps\log\left(\frac{1}{1-\abs w}\right),\\
	\le \abs w+2\eps\log\left(\frac{1}{1-\abs w}\right),
\end{cases}
$$
implying that 
$$
B\bb{0,r-2\eps\log\left(\frac{1}{1-r}\right)}\subset h \bigl(B(0,r)\bigr)\subseteq B\left(0,\,r+2\eps\log\left(\frac{1}{1-r}\right)\right).
$$
Define the function $\psi(r):=r-2\eps\log\left(\frac{1}{1-r}\right)$, for $r\in\sbb{0,1-2\eps}:=I_\eps$, and note that  $\psi$ is monotone increasing in $I_\eps$, and therefore
$$
\psi\bb{I_\eps}=\sbb{0,1-2\eps-2\eps\log\bb{\frac1{2\eps}}}=\sbb{0,1-2\eps\bb{1+\log\bb{\frac1{2\eps}}}}:=J_\eps,
$$
i.e., $\psi:I_\eps\rightarrow J_\eps$ is continuous and monotone increasing. In particular, for every $s\in J_\eps$
\begin{equation}\label{eq:psi}
	B\bigl(0,s\bigr)\subseteq h\bigl(B(0,\psi\inv(s))\bigr).
\end{equation}
A direct calculation shows that $\psi\inv:J_\eps\rightarrow I_\eps$ is defined by
\[
\psi^{-1}(s)
= 1 + 2\eps\, W_0\left(-\,\frac{e^{-\frac{1-s}{2\eps}}}{2\eps}\right),
\]
where $W_0$ denotes the principal branch of the Lambert $W$–function (the inverse of $t\mapsto t e^{t}$). Recall that $W_0(t)\le t$, then if $\abs z\in J_\eps$, following \eqref{eq:psi},
\[
\abs{h^{-1}(z)}\leq \psi^{-1}(\abs z) = 1 + 2\eps\, W_0\left(-\,\frac{e^{-\frac{1-\abs z}{2\eps}}}{2\eps}\right) 
\le 1 - e^{-\frac{1-\abs z}{2\eps}}.
\]
Let $R:=1-2\eps\bb{1+\log\bb{\frac1\eps}}$. 
For $|z|=R$ we have $e^{-\frac{1-\abs z}{2\eps}}=\frac\eps e$, implying that $\abs{h^{-1}(z)} \le 1 - \frac\eps e$, and since $h\inv(R\cdot\D)\subset B(0,1-\frac\eps e)$, applying \eqref{eq:Warsch}, we see that
$$
\M{\sup_{|z|<R}|h^{-1}(z)-z| \le \max_{|w|=1-\frac\eps e}|w-h(w)|\le 2\,\eps\log\left(\frac{1}{1-|w|}\right)\le 3\,\eps\log\left(\frac{1}{\eps}\right)}
$$
since $\eps<e^{-2}$, concluding the proof.
\end{proof}

\section{Approximating local behaviour}\label{sec:local}

The next lemma, which will later be used in the inductive step of the construction, `peels' one layer of attracting points while `straightening' the domain bringing it closer and closer to a disk. Formally, let $W$ be a bounded simply connected domain whose boundary, $\partial W$, is contained in the annulus $ \A\left({\frac1{1+\eps}},1+\varepsilon\right)$. Let $\varphi:\D\to W$ be a Riemann map with $\varphi(0)=0$. We construct an entire function, $g$, that approximates the inverse Riemann map, $\varphi^{-1}+\tau$, on $W$, where $\tau\in [2,\infty)$ is a prescribed translation parameter. Simultaneously, $g$ sends a prescribed collection of pairwise disjoint disks, $B(a_n,\delta)$, with centres $a_n$ in an annulus around $\partial W$, into what will later turn into an attracting basin, a small disk about $-A$, with $\tau\geq A\geq 2$; see Figure~\ref{fig:local_lemma}.


\begin{lem}[Local construction]\label{lem:small_disks} There exists $s\in\bb{0,1}$ so that for every $\kappa\in\bb{0,\frac{1}{5}}$, given
	\begin{enumerate}
		\item Parameters: a separation constant, 
		$0\leq \eta \leq \frac14$, a point, $-A$, and a translation value, $\tau$, with $\tau\geq A>2$.
		\item A conformal map, $\varphi:\D\rightarrow W$, satisfying that  $\varphi(0)=0$ and $\partial W\subset \A\bb{\frac1{1+\eps}, 1+\eps}$ for some $\eps <e^{-2}$ so that $3\eps\log\bb{\frac1\eps}<\kappa\cdot\eta$.
		\item A collection of points, $\bset{a_n}^N_{n=1}$, $N\in \N$, contained in $\A\bb{1+\eta-\eps,1+\eta+\eps}$ such that the disks $\bset{B\bb{a_n,\frac{\eta}5}}$ are pairwise disjoint.
	\end{enumerate}
	There exists an entire function $g:\C\rightarrow\C$ and a numerical constant $D>1$ such that for $$E_1:=\frac{D\cdot\tau}{\eta^{\frac32}\sqrt\kappa}\quad \text{ and } \quad E_2:=\exp\bb{-\frac\kappa2\cdot e^{\frac1\eta-\frac\kappa2}},$$
	\begin{enumerate}[label=$(L_{\arabic*})$]
		\item\label{item:att_basin} For every $\delta\in\bb{0,1}$ and for every $n\leq N,\; g\bb{B\bb{a_n, \frac{1}{10}\delta\cdot s\cdot\eta}}\subset B\bb{-A,E_1\cdot E_2\cdot\delta^{\frac 1{D}e^{\frac1{\eta}}-1}}$.
		\item\label{item:err_disk0} For every $\abs z\le1-\kappa\cdot\eta$,
		\begin{enumerate}
			\item \label{subitm:riemann} $\abs{g(z)-\bb{\varphi\inv(z)+\tau}}\le \frac{E_1}{\kappa}\cdot E_2^{\frac1{40}}.$
			\item \label{subitm:derivative} $\abs{g'(z)-\frac{\varphi'(0)}{\abs{\varphi'(0)}}}\le \frac{E_1}{\kappa^2\cdot\eta}\cdot E_2^{\frac1{40}}+\frac{12}{\kappa\cdot\eta}\cdot\eps\log\bb{\frac1\eps}$
		\end{enumerate}
		\item \label{item:growth} For every $R\geq 3$,
		$$
		M_g(R)=\underset{\abs z=R}\max\abs{g(z)}\le \tau+ R^{\frac D\eta e^{\frac1\eta}}.
		$$
	\end{enumerate}
\end{lem}


We devote the rest of the section to proving Lemma \ref{lem:small_disks},~ following the general scheme outlined in Subsection~\ref{subsec:method}.

\paragraph*{The model map}~\\
Fix a set of parameters as in the statement of Lemma \ref{lem:small_disks}, and observe that $B\left(0,1-\frac\kappa4\eta\right)\subset W$. Define the map $h\colon \C\to \C$ as
$$
h(z)=\begin{cases}
	\varphi\inv(z)+\tau,& z\in B\left(0,1-\frac\kappa4\eta\right),\\
	-A,& \text{otherwise.}
\end{cases}
$$
We set some notation for certain pairwise disjoint annuli that will appear in our construction (see Figure \ref{fig:local_lemma}):
\begin{equation}
	\begin{aligned}\label{eq_annuli}
		\mathcal{A}_1&=\A\bb{1-\frac\kappa2\eta,1-\frac\kappa4\eta},\quad &\mathcal{A}_2=& \A\bb{1+\frac\kappa4\eta,1+\frac\kappa2\eta},	\\	
		\mathcal{A}_3&=\A\bb{1+\frac{3}{5}\eta,1+\frac{7}{5}\eta}, \quad &\mathcal{A}_4=&\A\bb{1+2\eta,1+3\eta}.
	\end{aligned}
\end{equation} 

\begin{figure}[h]
	\centering
	\def\svgwidth{0,9\textwidth}
	\begingroup%
	\makeatletter%
	\providecommand\color[2][]{%
		\errmessage{(Inkscape) Color is used for the text in Inkscape, but the package 'color.sty' is not loaded}%
		\renewcommand\color[2][]{}%
	}%
	\providecommand\transparent[1]{%
		\errmessage{(Inkscape) Transparency is used (non-zero) for the text in Inkscape, but the package 'transparent.sty' is not loaded}%
		\renewcommand\transparent[1]{}%
	}%
	\providecommand\rotatebox[2]{#2}%
	\newcommand*\fsize{\dimexpr\f@size pt\relax}%
	\newcommand*\lineheight[1]{\fontsize{\fsize}{#1\fsize}\selectfont}%
	\ifx\svgwidth\undefined%
	\setlength{\unitlength}{2338.58267717bp}%
	\ifx\svgscale\undefined%
	\relax%
	\else%
	\setlength{\unitlength}{\unitlength * \real{\svgscale}}%
	\fi%
	\else%
	\setlength{\unitlength}{\svgwidth}%
	\fi%
	\global\let\svgwidth\undefined%
	\global\let\svgscale\undefined%
	\makeatother%
	\begin{picture}(1,0.61333333)%
		\lineheight{1}%
		\setlength\tabcolsep{0pt}%
		\put(0,0){\includegraphics[width=\unitlength,page=1]{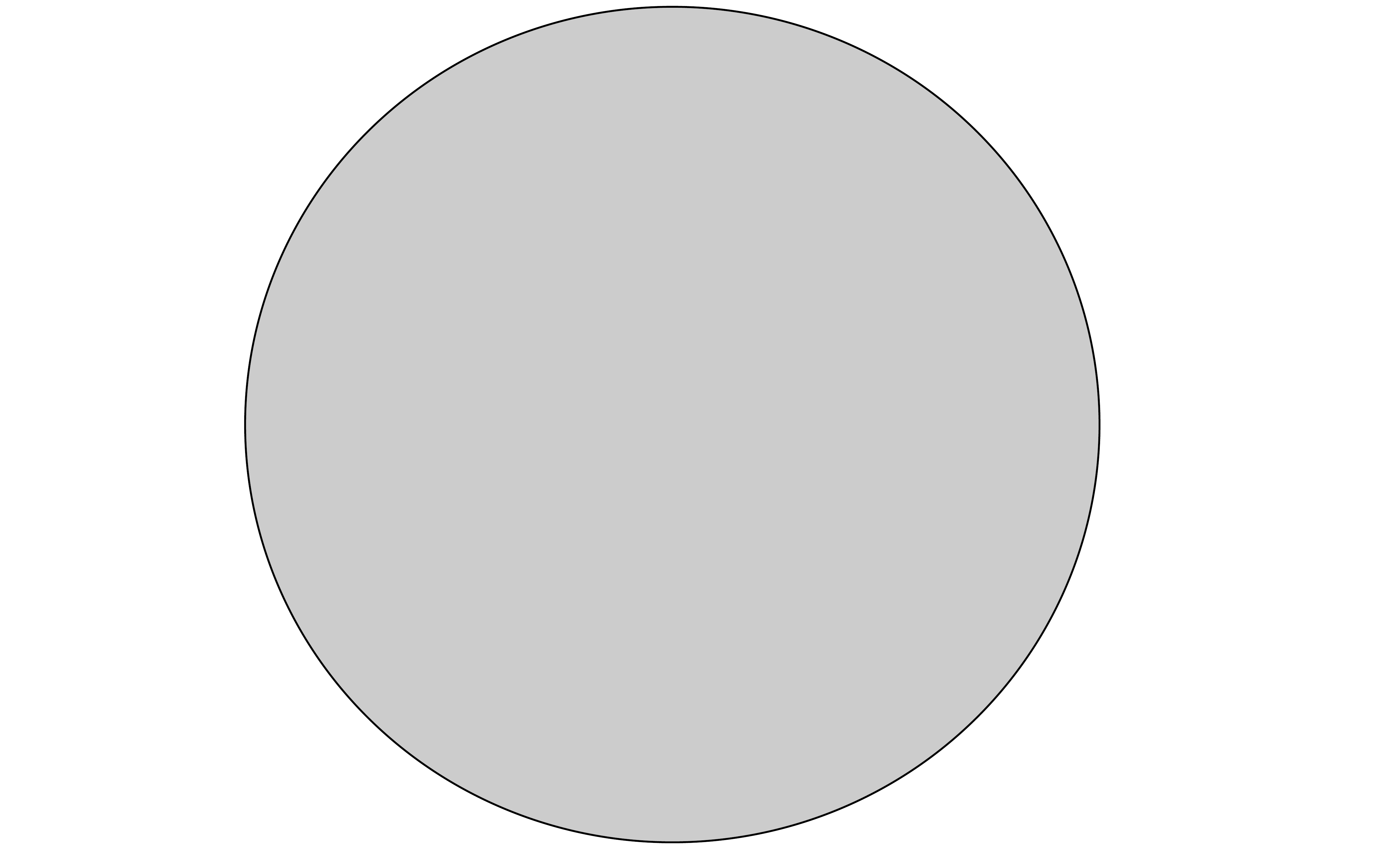}}%
		\put(0.47570439,0.2833586){\color[rgb]{0,0,0}\makebox(0,0)[lt]{\lineheight{1.25}\smash{\begin{tabular}[t]{l}$\fontsize{9pt}{1em}\mathcal{A}_4$\end{tabular}}}}%
		\put(0,0){\includegraphics[width=\unitlength,page=2]{local_lemma.pdf}}%
		\put(0.40868395,0.38443238){\color[rgb]{0,0,0}\makebox(0,0)[lt]{\lineheight{1.25}\smash{\begin{tabular}[t]{l}$\fontsize{9pt}{1em}\partial W$\end{tabular}}}}%
		\put(0,0){\includegraphics[width=\unitlength,page=3]{local_lemma.pdf}}%
		\put(0.56287311,0.25455171){\color[rgb]{0,0,0}\makebox(0,0)[lt]{\lineheight{1.25}\smash{\begin{tabular}[t]{l}$\fontsize{9pt}{1em}\mathbb{D}$\end{tabular}}}}%
		\put(0.37029018,0.50633046){\color[rgb]{0,0,0}\makebox(0,0)[lt]{\lineheight{1.25}\smash{\begin{tabular}[t]{l}$\fontsize{9pt}{1em}a_n$\end{tabular}}}}%
		\put(0.56335951,0.38916247){\color[rgb]{0,0,0}\makebox(0,0)[lt]{\lineheight{1.25}\smash{\begin{tabular}[t]{l}$\fontsize{9pt}{1em}\mathcal{A}_2$\end{tabular}}}}%
		\put(0,0){\includegraphics[width=\unitlength,page=4]{local_lemma.pdf}}%
		\put(0.801868,0.38479218){\color[rgb]{0,0,0}\makebox(0,0)[lt]{\lineheight{1.25}\smash{\begin{tabular}[t]{l}$\fontsize{9pt}{1em}\varphi^{-1}+\tau$\end{tabular}}}}%
		\put(0,0){\includegraphics[width=\unitlength,page=5]{local_lemma.pdf}}%
		\put(0.02562129,0.28694204){\color[rgb]{0,0,0}\makebox(0,0)[lt]{\lineheight{1.25}\smash{\begin{tabular}[t]{l}$\fontsize{9pt}{1em}-A$\end{tabular}}}}%
		\put(0.91857986,0.28647599){\color[rgb]{0,0,0}\makebox(0,0)[lt]{\lineheight{1.25}\smash{\begin{tabular}[t]{l}$\fontsize{9pt}{1em}\tau$\end{tabular}}}}%
		\put(0.63616191,0.41742427){\color[rgb]{0,0,0}\makebox(0,0)[lt]{\lineheight{1.25}\smash{\begin{tabular}[t]{l}$\fontsize{9pt}{1em}\mathcal{A}_3$\end{tabular}}}}%
		\put(0.71032425,0.46246849){\color[rgb]{0,0,0}\makebox(0,0)[lt]{\lineheight{1.25}\smash{\begin{tabular}[t]{l}$\fontsize{9pt}{1em}\mathcal{A}_4$\end{tabular}}}}%
		\put(0.49670207,0.35593203){\color[rgb]{0,0,0}\makebox(0,0)[lt]{\lineheight{1.25}\smash{\begin{tabular}[t]{l}$\fontsize{9pt}{1em}\mathcal{A}_1$\end{tabular}}}}%
	\end{picture}%
	\endgroup%
	
	\caption{Schematic of the construction in Lemma \ref{lem:small_disks} (not to scale). The transition region where $\bar\partial\chi\neq 0$ is confined to the two thin annuli $\mathcal{A}_1$ and $\mathcal{A}_2$. The darkest inner disk bounded by $\mathcal{A}_1$ is mapped by the model map to a neighbourhood of $\tau$, while the small disks are mapped to a neighbourhood of $-A$. The subharmonic weight, $u$, is obtained by gluing across $\mathcal{A}_4$.}\label{fig:local_lemma}
\end{figure}

\paragraph*{The subharmonic function}~\\
We apply the puncture lemma, Lemma \ref{lem:punctures}, with the disks $B\bb{a_n,\frac{\eta}5}, 1\le n\le N$, which are, by the choice of parameters, contained in the annulus $\mathcal{A}_3$, and with the underlying subharmonic function $$u_0\colon \mathbb{C}\to \R; \quad u_0(z):=e^{\frac{\abs z}\eta},$$ to create a subharmonic function $u_1\colon \mathbb{C}\to \R$. Note that $\Delta u_0(z)
= e^{\tfrac{|z|}{\eta}} \left( \frac{1}{\eta^{2}} + \frac{1}{\eta \lvert z \rvert} \right), 
$
where $\Delta$ denotes the Laplacian. Let $s:=e^{-25e^{2}/c}$, where $c$ is the numerical constant in Lemma \ref{lem:punctures}. Then, for every $n\leq N$ and $\delta\in\bb{0,1}$, applying part \eqref{item:punctures} of Lemma \ref{lem:punctures} to $\delta'=s\cdot \delta$, 
\begin{equation}\label{eq:max}
	\begin{aligned}
		\underset{z\in  B\bb{a_n,\delta'\cdot {\frac{\eta}5}}}\max\;u_1(z)&\le \underset{z\in B\bb{a_n,{\frac{\eta}5}}}\max\; u_0(z)- c\cdot \bb{{\frac{\eta}5}}^2\cdot \left(\underset{z\in B\bb{a_n,{\frac{\eta}5}}}\inf\; \Delta u_0\right)\log\bb{\frac1{\delta\cdot s}}\\
		&\le e^{{\frac1{ \eta}}\bb{1+\frac{7}{5}\eta}}-\frac {c\cdot\eta^2}{25}\cdot \frac{1}{\eta^2}e^{\frac1{\eta}}\log\bb{\frac1{\delta\cdot s}}\le e^{\frac1{\eta}}\bb{e^{2}-\frac{c}{25}\cdot\log\bb{\frac1{\delta\cdot s}}}\\
		&\le e^{\frac1{ \eta}}\bb{e^{2}-\frac{c}{25}\log\bb{\frac1{\delta}}-\frac{c}{25}\log\bb{\frac1{s}}}\le -\frac c{25}\cdot e^{\frac1{\eta}}\log\bb{\frac1\delta},
	\end{aligned}
\end{equation}
where in the second and third inequalities we use that $\eta<\frac14$, and the definition of $s$.

Next, define the subharmonic function $u\colon \C\to \R$ by
$$
u(z)=\begin{cases}
	u_1(z),& \abs z<1+2\eta,\\
	\max\bset{u_1(z),\frac{2e^3}\eta\cdot e^{\frac1\eta}\log\bb{\frac{\abs z}{1+2\eta}}},& 1+2\eta\le \abs z\le 1+3\eta,\\
	\frac{2e^3}\eta\cdot e^{\frac1\eta}\cdot\log\bb{\frac{\abs z}{1+2\eta}},& \text{otherwise}.
\end{cases}
$$
This map is indeed subharmonic by Corollary \ref{cor:gluing}: On the circle $\{|z|=1+2\eta\}$ note that $\frac{2e^3}\eta\cdot e^{\frac1\eta}\cdot \log\bb{\!\frac{|z|}{1+2\eta}}=0$, implying that $u=\max\{u_1,0\}=u_1$ as $\left.u_1\right|_{\{|z|=1+2\eta\}}\ge 0$. On the other hand, on $\bset{ \abs z=1+3\eta}$, 
$$
u_1(z)= e^{\frac1{\eta}\bb{1+3\eta}}= \frac{\frac{2e^3}\eta\cdot e^{\frac1\eta}\cdot\eta}2\le \frac{2e^3}\eta\cdot e^{\frac1\eta}\cdot \log\bb{1+\frac\eta{1+2\eta}}=\frac{2e^3}\eta\cdot e^{\frac1\eta}\cdot \log\bb{\frac{\abs z}{1+2\eta}},
$$
as long as  $\eta<\frac14$, i.e., $u=\max\bset{u_1,\M{\frac{2e^3}\eta\cdot e^{\frac1\eta}\log\bb{\frac{\abs z}{1+2\eta}}}}=\frac{2e^3}\eta\cdot e^{\frac1\eta}\cdot \log\bb{\frac{\abs z}{1+2\eta}}$.

\paragraph*{Bounding the integral}~\\
Following Proposition~\ref{prop:chi} with $\Omega_1=B(0,1-\frac{\kappa}{2}\eta)$ and $\Omega_2=\C\setminus B(0,1+\frac{\kappa}{2}\eta)$, that are $\kappa\cdot\eta$-separated
, let $\chi:\C\rightarrow[0,1]$ be a smooth function such that
\begin{enumerate}
	\item If $z\notin \A\bb{1-\frac\kappa 2\eta,1+\frac\kappa 2{\eta}}$, 
	then $\chi(z)=1$. 
	\item If $z\in \A\bb{1-\frac\kappa4\eta,1+\frac\kappa4\eta}$,
	then $\chi(z)=0$. 
	\item If $z\in \mathcal{A}_1\cup \mathcal{A}_2$, 
	then $\abs{\nabla\chi(z)}\leq \frac{C_0}{\kappa\cdot \eta}$ for some numerical constant $C_0>0$. 
\end{enumerate}

Observe that if $z\in \mathcal{A}_1\cup \mathcal{A}_2$, then, by definition, 
$$
u(z)=u_1(z)=e^{\frac1{\eta}\cdot\abs z}\ge \underset{\mathcal{A}_1\cup \mathcal{A}_2}\inf\; e^{\frac1{\eta}\abs z}=e^{\frac1{ \eta}-\frac\kappa2}, \quad \text{ and } \quad \vert h(z)\vert\leq \tau+2.
$$
Using this and noting that the area of $\mathcal{A}_1\cup \mathcal{A}_2$ is proportional to $\kappa\cdot \eta$, we have
\begin{align*}
	\iint_\C\abs{\bar\partial\chi(z)\cdot h(z)}^2e^{-u(z)}dm(z)&=\iint_{\mathcal{A}_1\cup \mathcal{A}_2}\abs{\bar\partial\chi(z)}^2\abs{h(z)}^2e^{-u(z)}dm(z)\\
	&\M{\le \frac{C_0^2}{\kappa^2\eta^2}\cdot \bb{\tau+2}^2 \cdot e^{-\bb{e^{\frac1{ \eta}-\frac\kappa2}}} \cdot m(\mathcal{A}_1\cup \mathcal{A}_2)=\frac{C_1^2}{\kappa\cdot \eta}\cdot \bb{\tau+2}^2 \cdot E_2^{\frac2\kappa}} \\
	& \leq \frac{\widetilde D^2\cdot\tau^2\cdot E_2^{\frac2\kappa}}{\kappa\cdot \eta}= \frac{s^2\cdot \eta^2}{10^4}\cdot E_1^2\cdot E_2^{\frac2\kappa}=:\mathcal{I}^2,
\end{align*}
for some numerical constant $\widetilde D$, where $s$ is the constant from \eqref{eq:max}; the factor $\bb{\frac s{100}}^2$ is included here for later convenience.

Define the function $g(z):=\chi(z)\cdot h(z)-{\beta}(z)$, where {$\beta$} is H\"ormander's solution to the $\bar\partial$-equation \M{for the function $\bar\partial\chi\cdot h$ and the subharmonic subharmonic function, $u$, defined above}. Following Proposition \ref{prop:method}, the map $g$ is entire. We shall now verify it satisfies properties \ref{item:att_basin}, \ref{item:err_disk0}, and \ref{item:growth} of the statement.

To see property \ref{item:att_basin}, let $z\in B(a_n, \frac{1}{10}\cdot\delta \cdot s\cdot \eta)$, then $B(z,  \frac{1}{10}\cdot\delta \cdot s\cdot \eta)\subset B(a_n,\delta \cdot s\cdot \frac\eta5)$, and by part \ref{item:error_bound} of Proposition~\ref{prop:method} with  $r=\frac{1}{10}\cdot\delta \cdot s\cdot \eta$ and \eqref{eq:max}, 
$$
\abs{g(z)-h(z)}\leq \frac{2+2\abs z^2}{{r}}\cdot {\mathcal{I}}\cdot \underset{w\in B(z,r)}\max\; e^{\frac12 u(w)}\le\frac{1}{\delta\cdot s\cdot \eta}\cdot s\cdot  \eta \cdot E_1\cdot E_2^{\frac1\kappa}\cdot e^{-\frac c{50}\cdot e^{\frac1{\eta}}\log\bb{\frac1\delta}}\le E_1\cdot E_2\cdot \delta^{\frac 1{D}\cdot e^{\frac1{\eta}}-1},
$$
since $\kappa<1$ and where $D=\frac{50}c$.

To see part (a) of property \ref{item:err_disk0} holds, let $z\in B(0,1-\frac{3\kappa}{4}\eta)$, and $r=\frac\kappa5{\eta}$, then $B(z,r)\subset B\bb{0,1-\frac{11\kappa}{20}{\eta}}$. In particular,  $h(w)=\varphi\inv(w)+\tau$ and $u(w)=u_0(w)$ for any $w\in B(z,r)$. Then, by part \ref{item:error_bound} or Proposition \ref{prop:method},

\begin{align}\label{eq:iia_extended}
	\abs{g(z)-\bb{\varphi\inv(z)+\tau}}&=\abs{g(z)-h(z)}\le \frac{5\bb{2+2\abs z^2}}{\kappa\cdot \eta}\cdot{\mathcal{I}}\cdot \underset{\abs w<1-\frac{11\kappa}{20}\eta}\max\; e^{\frac12 u(w)}\\
	&\le \frac{20}{\kappa\cdot \eta}\cdot \frac{s\cdot \eta\cdot E_1\cdot E_2^{\frac1\kappa}}{100}\cdot e^{\frac12e^{\frac1{\eta}\bb{1-\frac{11\kappa}{20}\eta}}}\le\frac{s}{4\kappa}\cdot E_1\cdot \exp\bb{-\frac12\cdot e^{\frac1\eta-\frac\kappa2}+\frac12e^{\frac1\eta}\cdot e^{-\frac{11\kappa}{20}}}\nonumber\\
	&=\frac s{4\kappa}\cdot E_1\cdot \exp\bb{-\frac12\cdot e^{\frac1\eta-\frac\kappa2}\bb{1-e^{-\frac{\kappa}{20}}}}\le\frac s{4\kappa}\cdot E_1\cdot \exp\bb{-\frac\kappa{80}\cdot e^{\frac1\eta-\frac\kappa2}}=\frac s{4\kappa}\cdot E_1\cdot E_2^{\frac1{40}},\nonumber
\end{align}
since $(1-e^{-x})>\frac x2$ whenever $x\in\bb{0,1}$.

To obtain good bounds on the derivative, namely, to prove part (b) of \ref{item:err_disk0}, we begin by approximating the derivative of $\varphi\inv$ using Lemma \ref{lem:riemann_approx}. Since this result may only be applied to a map with a positive derivative at zero, and this might not be the case for $\varphi\inv$, we rotate the function $\varphi$. Formally, if $\varphi'(0)=\abs{\varphi'(0)}e^{it_\varphi}$ we define $\Phi(z):=\varphi\bb{e^{-it_\varphi}\cdot z}$. Then $\varphi(\D)=\phi(\D)=W$ and 
$$
\Phi(0)=\varphi\bb{0}=0,\quad\quad \Phi'(0)=e^{-it_\varphi}\varphi'\bb{0}=e^{-it_\varphi}\cdot e^{it_\varphi}\cdot\abs{\varphi'(0)}=\abs{\varphi'(0)}>0,\quad\quad\Phi\inv(w)=e^{it_\varphi}\varphi\inv(w).
$$
Since, by assumption,
$$
\partial W\subset \A\left({\frac1{1+\eps}},1+\varepsilon\right),
$$
we may use Lemma \ref{lem:riemann_approx} to conclude that if $\abs z<1-\kappa\cdot\eta<1-3\eps\log\bb{\frac1\eps}$, then
$$
\abs{\varphi\inv(z)-e^{-it_\varphi}\cdot z}=\abs{e^{it_\varphi}\varphi\inv(z)- z}=\abs{\Phi\inv(z)- z}\le 3\eps\log\bb{\frac1{\eps}}.
$$

Next, we will use Cauchy's formula for the derivative: Given any two functions $F,G$ holomorphic on a disk $B(z,r)$ for some $z\in \C$ and $r>0$,
\begin{equation}\label{eq:derivative}
	\abs{F'(z)-G'(z)}=\abs{\frac1{2\pi i}\integrate{\abs{z-w}=r}{}{\frac{F(w)}{(z-w)^2}}w-\frac1{2\pi i}\integrate{\abs{z-w}=r}{}{\frac{G(w)}{(z-w)^2}}w}\le \frac{\underset{\abs{z-w}=r}\max\abs{F-G}}r.
\end{equation}
We apply \eqref{eq:derivative} to the disks $B\bb{z,\frac{\kappa\cdot \eta}4}$ which, whenever $\abs z<1-\kappa\cdot \eta$, satisfy $B\bb{z,\frac{\kappa\cdot \eta}4}\subset\bset{\abs w<1-\frac{3\kappa}{4}\eta}$, combined with part (a) of \ref{item:err_disk0}, which we already proved in \eqref{eq:iia_extended} to hold for all the points in the latter set,
\begin{align*}
	\abs{g'(z)-\frac{\varphi'(0)}{\abs{\varphi'(0)}}}&=\abs{g'(z)-e^{-it_\varphi}}=\abs{g'(z)-\bb{e^{-it_\varphi}\cdot z+\tau}'(z)}\\
	&\le \frac {4}{\kappa\cdot \eta}\bb{\underset{\abs{w-z}=\frac{\kappa\cdot \eta}4}\sup\abs{g-\bb{\varphi\inv+\tau}}+\underset{\abs{w-z}=\frac{\kappa\cdot \eta}4}\sup\abs{\varphi\inv(w)-e^{-it_\varphi}\cdot w}}\\
	&\le \frac{E_1}{\kappa^2\cdot\eta}\cdot E_2^{\frac1{40}}+\frac{12}{\kappa\cdot\eta}\cdot\eps\log\bb{\frac1\eps}.
\end{align*}

Lastly, to see property \ref{item:growth}, let $z$ such that $\abs z>3$. Then $\vert h(w)\vert =\vert -A\vert \leq \tau$ and $u(w)=\frac{2e^3}\eta\cdot e^{\frac1\eta}\cdot \log\bb{\frac{\abs w}{1+2\eta}}$ for all $w\in B(z,1)$. Thus, using part \ref{item:growth_bound} in Proposition \ref{prop:method},
$$
\abs{g(z)}\leq \tau+ \frac{2+2\abs z^2}{{\sqrt \pi}}\cdot {\mathcal{I}}\cdot \underset{w\in B(z,r)}\max\; e^{\frac12 u(w)}\le \tau+ 3\abs z^2\cdot \frac{s\cdot \eta}{100}\cdot E_1\cdot E_2\cdot \bb{1+\abs z}^{\frac{e^3}\eta\cdot e^{\frac1\eta}}\le \tau+ \abs z^{\frac D\eta e^{\frac1\eta}}
$$
if $\eta<\frac14$ and $D:=e^3+1$, concluding the proof.
\section{Constructing a sequence of entire functions}
For the rest of the paper, let us fix $\alpha\in (0,1)$ and a simply connected domain $U\subset \C$ with analytic boundary, as in the statement of Theorem \ref{thm:intro}. Note that since $U$ has analytic boundary, then must be regular. By applying an affine transformation, we may assume without loss of generality that $0\in U$. Let $\varphi_U:\D\rightarrow U$ be a Riemann map normalised such that $\varphi_U(0)=0$. By analyticity of $\partial U$, $\varphi_U$ has a conformal extension to a larger disk $\mathcal{D}\supset\D$, see e.g., \cite[Proposition~3.1]{pommerenke_boundary}. By abuse of notation, we also denote this extension by $\varphi_U$. It will be useful to consider both the map $\varphi_U$ and its inverse,
\begin{equation*}
	\varphi^{-1}_U\colon \varphi_U(\mathcal{D})\rightarrow  \mathcal{D} \quad \text{ and } \quad \varphi_U\colon \mathcal{D}\rightarrow  \varphi_U(\mathcal{D}).
\end{equation*} 

The goal of this section is to construct a sequence of entire maps whose limit is the desired entire map realising $U$ as a univalent wandering domain. Our construction requires several parameters. While there is some freedom in choosing those parameters, they do need to satisfy certain relations.

\subsection{Setting the Stage}
Let $\bset{\tau_k}$ be a monotone increasing sequence of numbers, $\tau_1>\tau_1(\alpha)$, satisfying


\begin{center}
	\begin{minipage}[b]{.4\textwidth}
		\vspace{-\baselineskip}
		\begin{equation*}
			\tag{C1}\label{cond:2} \tau_{k+1}^{11}\le\exp\bb{\tau_k^{\frac\alpha7}}
		\end{equation*}
	\end{minipage}\quad\quad
	and
	\begin{minipage}[b]{.4\textwidth}
		\vspace{-\baselineskip}
		\begin{equation*}
			\tag{C2}\label{cond:1} \frac{\log\log\tau_{k+1}}{\log\log\tau_{k+2}}\leq\frac12
		\end{equation*}
	\end{minipage}
\end{center}
\
\begin{rmk}
	Examples of sequences $\bset{\tau_k}$ satisfying the requirements are:
	\begin{itemize}
		\item $\tau_{k}=\exp\bb{\log^2\bb{\tau_{k-1}}}=\exp\bb{\exp\bb{2^{k+k_1}}}$ for some $k_1\in \N$ large enough.
		\item $\tau_{k}=\exp\bb{\iota\cdot \tau_{k-1}^\alpha}$ for $0<\iota\le \frac1{11}$ and $\tau_1>1$.
	\end{itemize}
\end{rmk}

%

Next, define, for each $k\geq 0$,
$$
\rho_k:=\frac1{\log\bb{\log\bb{\tau_{k+1}}}}.
$$
Note that since $\tau_k\nearrow \infty$ as $k\to \infty$,  then $\rho_k\searrow 0$, i.e., it is a monotone decreasing sequence converging to zero.

Throughout this section we work with a fixed monotone increasing sequence \(\bset{\tau_k}\) satisfying \eqref{cond:2} and \eqref{cond:1}. The next proposition collects routine bounds that hold if \(\tau_1\) is chosen sufficiently large. We will use these estimates without further comment; when increasing \(\tau_1\) is required, this will be explicitly indicated at the relevant point.
\begin{prop} \label{prop:tau1} We choose  $\tau_1$ large enough so that the following hold:
	\begin{enumerate}[label=(\alph*)]
		\item\label{item:tau1_a} For every $j \geq 1$, $\sum_{\ell=j}^{\infty}\frac{1}{\tau_{\ell+1}}\le \frac{2}{\tau_{j+1}}< \frac{\log^2(\tau_j)}{\tau_j}<\frac{\log^2(\tau_1)}{\tau_1}\le 1.$
		\item\label{item:tau1_b}  For all $k\geq 1$, we have $\frac{3^k}{\tau_k}\log^2\bb{\tau_k}<\rho_k$.
		\item\label{item:tau1_c}  $\rho_0$ is small enough such that $\D_{1+\rho_0}\subset \mathcal{D}$, implying that $\varphi_U$ is well defined as a conformal map on $\D_{1+\rho_0}$.
		\item \label{item:tau1_e} Any map $F\colon  \varphi_U(\mathcal{D})\to \C$ such that $\vert F(z)-(\varphi_U^{-1}(z)+\tau_k)\vert<\frac{1}{\tau_1}$ for  some $k$ and all $z\in \varphi_U(\mathcal{D})$ is univalent on $\overline{\D_{1+\rho_0}}$. 
	\end{enumerate} 
\end{prop}
\begin{proof}
	First, note that if \eqref{cond:1} holds then ${\tau_{k+2}}\ge\exp\bb{\log^2(\tau_{k+1})}$, which means that for every $n\in\N$ if we choose $\tau_1\ge\tau_1(n)$, then $\tau_{k+1}\ge\tau_k^n$ for all $k$. By induction this would imply that for every $k$ and $j$ we have
	\begin{equation}\label{eq:exp_growth}
		\tau_{k+j}\ge\tau_{k+j-1}^n\ge \tau_{k+j-2}^{n^2}\ge\cdots\ge\tau_k^{n^j}\ge\tau_1^{n^{j+k}},
	\end{equation}
	therefore, if $\tau_j\ge\tau_1\ge 2$ and we take $n=4$,
	$$
	\sum_{\ell=j}^{\infty}\frac{1}{\tau_{\ell+1}}=\sum_{\ell=1}^{\infty}\frac{1}{\tau_{j+\ell}}\le \sum_{\ell=1}^{\infty}\frac{1}{\tau_j^{2\ell}}\le \frac{1}{\tau_j}\sum_{\ell=0}^{\infty}2^{-\ell}<\frac{2}{\tau_{j+1}}.
	$$
	Next, if $\tau_j\ge e^2$, then
	$$
	\frac{2}{\tau_{j+1}}\le {\frac{2}{\tau_j^4}\le} \frac{\log^2(\tau_j)}{\tau_j}<\frac{\log^2(\tau_1)}{\tau_1}\le 1.
	$$
	
	This proves \ref{item:tau1_a}. To prove \ref{item:tau1_b}, {note that} 
	$$
	\frac{3^k}{\tau_k}\log^2\bb{\tau_k}\;\le\; \rho_k \iff 3^k\le\frac{\tau_k}{\log^2(\tau_{k})\log\log(\tau_{k+1})}.
	$$
	
	However, following \eqref{cond:2}, $$
	\frac1{\rho_k}{= }\log\log(\tau_{k+1})
	\le \log{\Bigl(\frac1{11}\tau_k^{\frac\alpha7}\Bigr)}
	\le \frac{\alpha}{{7}}\log(\tau_k)
	<\log\bb{\tau_k}
	$$
	for $\alpha\in\bb{0,1}$, implying that
	$$
	\frac{\tau_k}{\log^2(\tau_{k})\log\log(\tau_{k+1})}\ge \frac{\tau_k}{\log^3(\tau_k)}\ge \sqrt{\tau_k}\ge\tau_1^{2^{k-1}}\ge 3^k,
	$$
	following \eqref{eq:exp_growth} if $\tau_1$ is large enough. 
	
	Item \ref{item:tau1_c} holds if $\rho_0$ is small enough, which corresponds to $\tau_1$ being large enough as a function of the domain,~$U$. Finally, for the proof of \ref{item:tau1_e}, see for example \cite[Lemma 2.3]{MartiPete_JAMS_2025}.
\end{proof}
\paragraph*{More Parameters: } For every $k\geq 0$, define
$$r_k=1+\frac12\bb{\rho_{k+1}+\rho_k}=1+\rho_{k+1}+\frac12\bb{\rho_{k}-\rho_{k+1}},$$
and  the sets
$$U_k=\varphi_U\bb{\D_{1+\rho_k}}\quad \text{ and } \quad V_{k+1}=\varphi_U\bb{\D_{r_{k}}}.
$$
Note that for every $k\ge 0$, $U_{k+1} \subset V_{k+1}\subset U_k.$

As the function $\left.\varphi_U\right|_{\overline{\mathcal D}}$ is univalent, then $\left.\varphi_U'\right|_{\overline{\mathcal D}}$ is continuous, and non-vanishing, i.e., by rescaling the set $U$, we may fix a constant $C_U\ge 2$ such that
\begin{equation}\label{eq:derivative_phiU}
	\frac12 \;\le\; C_U^{-1}
	\;\le\; \inf_{z\in \overline U_0}\bigl|(\varphi_U^{-1})'(z)\bigr|
	\;\le\; \sup_{z\in \overline U_0}\bigl|(\varphi_U^{-1})'(z)\bigr|
	\;\le\; C_U.
\end{equation}

Note that \(U\subset U_k\) for all \(k\) (see Figure~\ref{fig:induction}). For every \(k\ge0\) we choose a collection of points
\begin{equation}\label{eq:points_Qk}
	Q_k:=\{a_j^k\}^{N_k}_{j=0}\subset \partial U_k,\quad N_k\in\mathbb{N},\qquad
	\frac1k\text{ - dense and \((r_k\cdot \rho_k)\)-separated on }\partial U_k.
\end{equation}
This can be arranged by taking equally spaced points on \(\partial\D_{1+\rho_k}\) and setting $Q_k$ to be their image under $\varphi_U$; by \eqref{eq:derivative_phiU}, $\varphi_U$ is bi-Lipschitz in a neighbourhood of this set, and so both density and separation transfer (up to uniform rescaling), and \(\partial U\) is the accumulation set of \(Q_k\) as \(k\to\infty\).

We are now ready to state our construction theorem.
\begin{thm}\label{thm:sequence}
	There exists  a sequence of entire functions $f_k:\C\rightarrow\C$ , $k\geq 1$, satisfying:
	\begin{enumerate}[label=$(S_\alph*)$]
		\item\label{itm:consistency} For every $k\ge 2$, $\underset{\abs z\le 2\tau_{k-1}}\sup\;\abs{f_k(z)-f_{k-1}(z)}<\frac1{\tau_{k+1}}$.
		\item\label{itm:growth} If $\abs z\ge2\tau_k$ then $\abs{f_k(z)}^2 e^{-\abs z^\alpha}<	15\cdot \abs z^4\tau_k^6$.
		\item\label{itm:iterations} 
		\begin{enumerate}[label=(\roman*)]
			\item \label{subitm:first} If $z\in\overline{U_2}$ then $\abs{f_1(z)-\bb{\varphi_U\inv(z)+\tau_1}}<\frac1{\tau_2}$.
			\item \label{subitm:iterations} If $k\ge2$, $j\le k-1$ and $\nu\le j+1$ then $z\in\overline{U_{\nu+1}}$ satisfies $\abs{f_k^\nu(z)-f_j^\nu(z)}<\frac{2\cdot 3^\nu}{\tau_{j+2}}$.
			\item \label{subitm:final} If $z\in\overline{U_{k+1}}$ then $\abs{f_k^{k+1}(z)-\bb{\varphi_U\inv(z)+\tau_{k+1}}}<\frac{3^{k+1}}{\tau_{k+1}}$.
		\end{enumerate}
		
		\item\label{itm:derivative}
		\begin{enumerate}[label=(\roman*)]
			\item $C_U\inv -\sumit j 1 k\frac{1}{\tau_{j+1}}<\left.\abs{f'_k}\right|_{\overline{U_2}}<C_U+\sumit j 1 k\frac{1}{\tau_{j+1}}$.
			\item \label{subitem:derivative_iterations}If $w\in B\bb{f_k^{\nu}(z),\frac{\rho_{\nu}}{10}-\sumit\ell\nu k\frac{3^\ell}{\tau_{\ell+1}}}$ for some $z\in \overline{U_{\nu+1}}$ and $\nu\leq k$, then $\abs{f'_k\bb{w}-1}<\frac{\log^2\bb{\tau_\nu}}{2\tau_{\nu}}+\sumit \ell {\nu} k\frac{1}{\tau_{\ell+1}}.$
		\end{enumerate}
		\item\label{itm:attracting} $f_1(B\bb{-\tau_1,1})\subset B\bb{-\tau_1, \frac{1}{\tau_1}}$ and for every $k\geq 2$ and $ 1\le j\le N_k$, $f_k\bb{B\bb{f^k_{k-1}\bb{a_j^k}, \rho_k^2}}\subset B\bb{-\tau_1,\frac1{2}}$.
	\end{enumerate}
\end{thm}

\paragraph*{Outline of the proof.}
We construct the sequence recursively.

\emph{Base step ($k=1$).} The first step is slightly different, since Lemma \ref{lem:small_disks} only applies to domains that are \textit{almost round}, which {might not be the case} for the domains $U_k.$ Thus, in this step we \emph{straighten} $U$ and the supersets $U_k$ with the map $\varphi_U$, turning $\partial U_k$ to be uniformly close to a circle (up to a fixed bi-Lipschitz distortion). We then pick the initial subharmonic weight $u_1$ from Proposition~\ref{prop:sh_u1} below, which is strongly negative exactly where accuracy is needed and grows like $|z|^{\alpha}$ elsewhere. With this geometry and $u_1$ in place, we glue the first model map using a cutoff function, $\chi_1$, supported on thin annuli and the narrow collar $U_0\setminus V_1$, and correct the resulting $\bar\partial$–error using the method described in Subsection \ref{subsec:method}, obtaining $f_1$ with properties in \ref{itm:growth}-\ref{itm:attracting}; see Figure~\ref{fig:first_function}.

\emph{Inductive step ($k\ge2$).} From now on the geometry is normalized: we work on `almost–round' disks. On the almost–round disk containing $f_{k-1}^k(U_{k+1})$, we apply the Local Lemma, Lemma~\ref{lem:small_disks}, to build a model map that sends small disks near the images $f_{k-1}^k(a_j^k)$ into a neighbourhood of $\bb{-\tau_1}$, which by construction will be contained in an attracting basin, and (after the standard rescaling), approximates $\varphi_U^{-1}(z)+\tau_{k+1}$ on $U_{k+1}$; see Figure \ref{fig:induction}. We choose a cut-off functions, $\chi_k$, with $\bar\partial\chi_k$ supported in thin annuli and use the corresponding weight $u_k$ (a $|z|^{\alpha}$–type weight punctured at $0$ and at $\tau_k$). The uniform control of `roundness' across stages and the growth of $\tau_k^{\alpha}$ make the $\bar\partial$–correction increasingly small, so the error and derivative bounds improve with $k$. This yields properties \ref{itm:consistency}-\ref{itm:attracting} for $f_k$ and completes the construction.

\subsection{Proof of Theorem \ref{thm:sequence} for $k=1$:}
\paragraph*{The model map}~\\
Define the model map
\begin{equation}\label{eq:h1}
	\begin{aligned}
		h_1(z)=	\begin{cases}
			-\tau_1,& z\in B\bb{-\tau_1,\frac{\tau_1}{3}},\\
			\varphi_U\inv(z)+\tau_{1},& z\in U_0,\\
			z+\bb{\tau_2-\tau_1},& z\in B\bb{\tau_1,\frac{\tau_1}3},\\
			0,& \text{otherwise}.
		\end{cases}
	\end{aligned}
\end{equation}
For the rest of the section, we will denote 
$$A_U:=\diam(U_0)+1, $$
and define the annuli
\begin{equation}\label{eq_annuli_f1}
	\mathcal{A}_1=\left\{\,z\in\mathbb{C}\colon  \frac{\tau_1}{4}<|z+\tau_1|<\frac{\tau_1}{3}\right\} \quad\text{ and } \quad  \mathcal{A}_2= \left\{\,z\in\mathbb{C}\colon  \frac{\tau_1}{4}<|z-\tau_1|<\frac{\tau_1}{3}\right\}.
\end{equation}

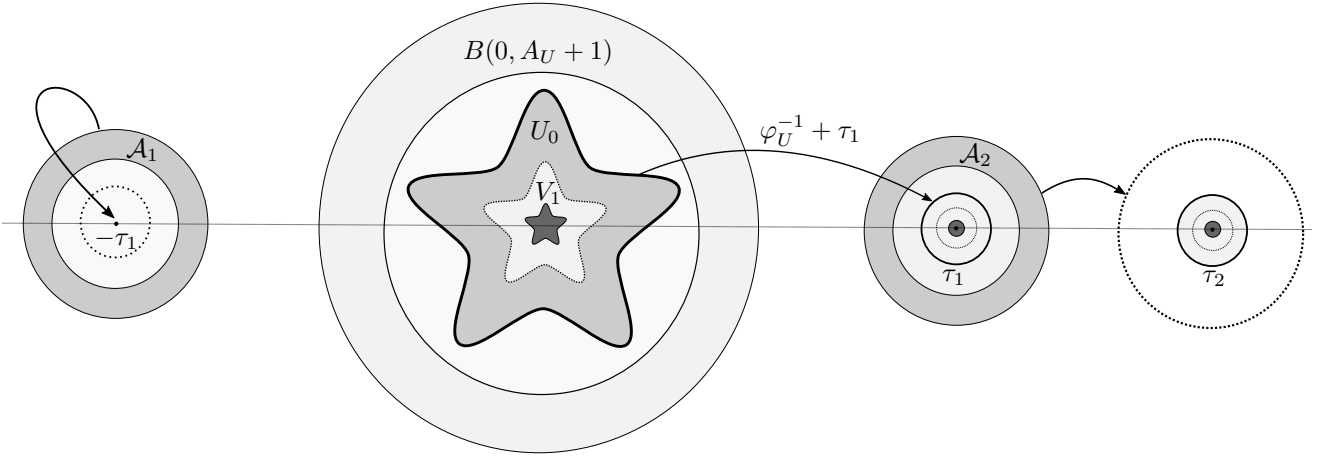
\begin{figure}[h]
	\centering
	\def\svgwidth{\textwidth}
	\begingroup%
	\makeatletter%
	\providecommand\color[2][]{%
		\errmessage{(Inkscape) Color is used for the text in Inkscape, but the package 'color.sty' is not loaded}%
		\renewcommand\color[2][]{}%
	}%
	\providecommand\transparent[1]{%
		\errmessage{(Inkscape) Transparency is used (non-zero) for the text in Inkscape, but the package 'transparent.sty' is not loaded}%
		\renewcommand\transparent[1]{}%
	}%
	\providecommand\rotatebox[2]{#2}%
	\newcommand*\fsize{\dimexpr\f@size pt\relax}%
	\newcommand*\lineheight[1]{\fontsize{\fsize}{#1\fsize}\selectfont}%
	\ifx\svgwidth\undefined%
	\setlength{\unitlength}{2715.59055118bp}%
	\ifx\svgscale\undefined%
	\relax%
	\else%
	\setlength{\unitlength}{\unitlength * \real{\svgscale}}%
	\fi%
	\else%
	\setlength{\unitlength}{\svgwidth}%
	\fi%
	\global\let\svgwidth\undefined%
	\global\let\svgscale\undefined%
	\makeatother%
	\begin{picture}(1,0.34655532)%
		\lineheight{1}%
		\setlength\tabcolsep{0pt}%
		\put(0,0){\includegraphics[width=\unitlength,page=1]{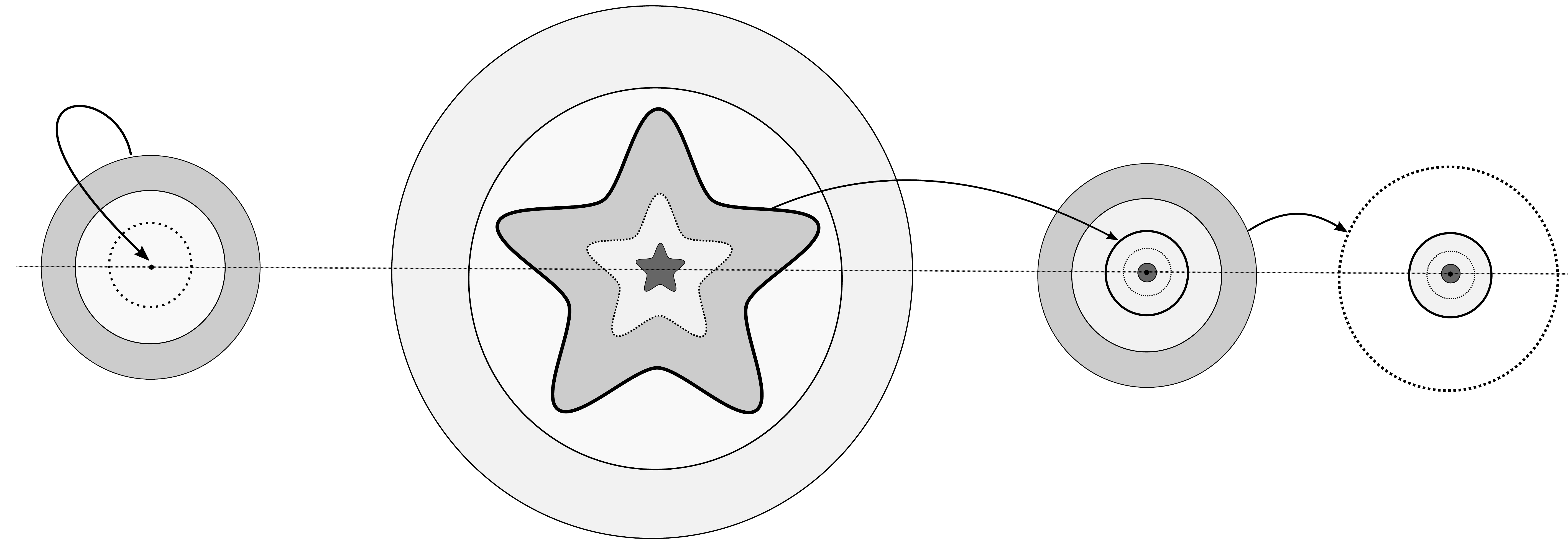}}%
		\put(0.35892444,0.30097894){\color[rgb]{0,0,0}\makebox(0,0)[lt]{\lineheight{1.25}\smash{\begin{tabular}[t]{l}$\fontsize{9pt}{1em}B(0,A_U+1)$\end{tabular}}}}%
		\put(0.40932982,0.23984057){\color[rgb]{0,0,0}\makebox(0,0)[lt]{\lineheight{1.25}\smash{\begin{tabular}[t]{l}$\fontsize{9pt}{1em}U_0$\end{tabular}}}}%
		\put(0.41284514,0.19418971){\color[rgb]{0,0,0}\makebox(0,0)[lt]{\lineheight{1.25}\smash{\begin{tabular}[t]{l}$\fontsize{9pt}{1em}V_1$\end{tabular}}}}%
		\put(0.7207967,0.13279838){\color[rgb]{0,0,0}\makebox(0,0)[lt]{\lineheight{1.25}\smash{\begin{tabular}[t]{l}$\fontsize{9pt}{1em}\tau_1$\end{tabular}}}}%
		\put(0.91764225,0.13199166){\color[rgb]{0,0,0}\makebox(0,0)[lt]{\lineheight{1.25}\smash{\begin{tabular}[t]{l}$\fontsize{9pt}{1em}\tau_2$\end{tabular}}}}%
		\put(0.08007614,0.16005886){\color[rgb]{0,0,0}\makebox(0,0)[lt]{\lineheight{1.25}\smash{\begin{tabular}[t]{l}$\fontsize{9pt}{1em}-\tau_1$\end{tabular}}}}%
		\put(0.10409366,0.22671601){\color[rgb]{0,0,0}\makebox(0,0)[lt]{\lineheight{1.25}\smash{\begin{tabular}[t]{l}$\fontsize{9pt}{1em}\mathcal{A}_1$\end{tabular}}}}%
		\put(0.73369143,0.22273475){\color[rgb]{0,0,0}\makebox(0,0)[lt]{\lineheight{1.25}\smash{\begin{tabular}[t]{l}$\fontsize{9pt}{1em}\mathcal{A}_2$\end{tabular}}}}%
		\put(0.58295665,0.24059833){\color[rgb]{0,0,0}\makebox(0,0)[lt]{\lineheight{1.25}\smash{\begin{tabular}[t]{l}$\fontsize{9pt}{1em}\varphi^{-1}_U+\tau_1$\end{tabular}}}}%
	\end{picture}%
	\endgroup%
	
	\caption{First step: straightening using the Riemann map $\varphi_U\inv$, the domains are straightened so that $\partial U$ becomes circular and collars become annuli. The dark grey region indicates $\supp(\bar\partial\chi)$, confined to the thin annuli, $\mathcal{A}_1,\mathcal{A}_2$, and a narrow collar, $U_0\setminus V_1$. The subharmonic weight, $u_1$, is very negative on the regions where we require accurate approximation (namely on $B(\pm\tau_1,4)$ and on $\overline{U_1}$), while outside these sets it behaves essentially like $|z|^{\alpha}$. Thus the gluing occurs where $\bar\partial\chi_1\neq 0$ (the dark areas in the figure), which is where $u_1$ is large, away from the approximation zones where $u_1$ is small and controls the contribution of the $\bar\partial$–error.}\label{fig:first_function}
\end{figure}

\paragraph*{The subharmonic function}
\begin{prop} \label{prop:sh_u1}There exists a subharmonic function $u_1\colon \C\to \R$ with the following properties:
	\begin{enumerate}
		\item \label{item:prop_sh_def} If $z\nin B\bb{-\tau_1,\frac{\tau_1}{4}}\cup B\bb{0,A_U+1}\cup B\bb{\tau_1,\frac{\tau_1}4}$, then $u_1(z)=\tau_1\cdot  \bb{\abs z-A_U}^{\alpha}$.
		\item\label{item:prop_sh_bound2} If $z\in \overline{U_0}\setminus V_1$, then $u_1(z)>0$.
		\item \label{item:prop_sh_bound1} If $z\in  B(-\tau_1, 4) \cup \overline{U}_1 \cup  B(\tau_1, 4)$, then $u_1(z)\leq -3\tau_1^\alpha$.
	\end{enumerate}
\end{prop}
\begin{proof}
	Denote $\phi_1:=\varphi_U^{-1}\vert_{U_0}\colon U_0\rightarrow\D_{1+\rho_0}$  and let $\phi_2:\widehat{\C}\setminus V_1\rightarrow\widehat{\C}\setminus \overline{\D}$ be a Riemann map. Fix $\tilde{r}\in \left(1+\rho_1, r_0\right)$, and define the map $v_1\colon \C\to \R$ as
	$$
	v_1(z):=	\begin{cases}
		C_1 \log\bb{\frac{\abs{\phi_1(z)}}{\tilde{r}}},& z\in V_1,\\
		\max\bset{C_1\cdot\log\bb{\frac{\abs{\phi_1(z)}}{\tilde{r}}} ,C_2\cdot\log\bb{\abs{\phi_2(z)}}},& z\in \overline{U_0}\setminus V_1,\\
		C_2\cdot \log\bb{\abs{\phi_2(z)}},& z\in B(0,A_U)\setminus U_0,\\
		\max\bset{C_2\cdot \log\bb{\abs{\phi_2(z)}},\tau_1\bb{\abs z-A_U}^{\alpha}},& z\in B(0,A_U+1)\setminus B(0,A_U),\\
		\tau_1\bb{\abs z-A_U}^{\alpha},& z\in\C\setminus B(0,A_U+1),
	\end{cases}
	$$
	where the constants $C_1$ and $C_2$ will be chosen momentarily to guarantee that the map $v_1$ is subharmonic. We will prove this fact using Corollary \ref{cor:gluing}. We start by choosing $C_1:= \frac{3\tau_1^\alpha}{\log\bb{\frac{\tilde{r}}{1+\rho_1}}}
	=C(\tau_1)>0$. Note that
	$$
	\underset{z\in \overline{U_1}}\max\; v_1(z)<C_1\log\bb{\frac{1+\rho_1}{\tilde{r}}}{=} \frac{3\tau_1^\alpha}{\log\bb{\frac{\tilde{r}}{1+\rho_1}}}\cdot \log\bb{\frac{1+\rho_1}{\tilde{r}}}=-3\tau_1^\alpha.
	$$
	
	Next, note that for every $z\in \overline{U_0}\setminus V_1$, 
	\begin{equation}\label{eq:U_1}
		C_1\cdot\log\bb{\frac{\abs{\phi_1(z)}}{\tilde{r}}}\geq C_1 \log\bb{\frac{r_0}{\tilde{r}}}\geq C_1>0,
	\end{equation}
	while if $z\in \partial V_1$, by definition, $\abs{\phi_2(z)}=1$, implying that $C_1\log\bb{\frac{\abs{\phi_1(z)}}{\tilde{r}}}>0=C_2\log\bb{\abs{\phi_2(z)}}=0$, no matter what $C_2$ is. On the other hand, if we choose $C_2:=\frac{2C_1\log\bb{\frac{1+\rho_0}{\tilde{r}}}}{\underset{z\in\partial U_0}\inf\; \log\bb{\abs{\phi_2(z)}}}$, then for every $z\in\partial U_0$,
	$$
	C_1\log\bb{\frac{\abs{\phi_1(z)}}{\tilde{r}}}=C_1\log\bb{\frac{1+\rho_0}{\tilde{r}}}=\frac{C_2}{2}\cdot \underset{w\in\partial U_0}\inf\; \log\bb{\abs{\phi_2(w)}}<C_2\cdot \underset{w\in\partial U_0}\inf\; \log\bb{\abs{\phi_2(w)}}\le C_2\log\bb{\abs{\phi_2(z)}}.
	$$
	We conclude that $v_1$ is subharmonic in $B(0,A_U)$. 
	
	Similarly, for every $\abs z=A_U$, 
	$$
	\tau_1\bb{\abs z-A_U}^{\alpha}=0<C_2\log\bb{\abs{\phi_2(z)}},
	$$
	no matter what $\tau_1>0$ is. Lastly, for every $\abs z=\bb{A_U+1}$,
	$$
	\tau_1\bb{\abs z-A_U}^{\alpha}=\tau_1>C_2\cdot \underset{\abs z=A_U+1}\sup\;\log\bb{\abs{\phi_2(z)}},
	$$
	if $\tau_1$ is large enough since
	\begin{align*}
		\tau_1> C_2\cdot \underset{\abs z=A_U+1}\sup\;\log\bb{\abs{\phi_2(z)}}=2\cdot\frac{C_1\log\bb{\frac{1+\rho_0}{\tilde{r}}}}{\underset{z\in\partial U_0}\inf\; \log\bb{\abs{\phi_2(z)}}} \underset{\abs z=A_U+1}\sup\;\log\bb{\abs{\phi_2(z)}}\\
		=6\tau_1^\alpha\cdot \frac{\log\bb{\frac{1+\rho_0}{\tilde{r}}}}{\log\bb{\frac{\tilde{r}}{1+\rho_1}}}\cdot \frac{\underset{\abs z=A_U+1}\sup\;\log\bb{\abs{\phi_2(z)}}}{\underset{z\in\partial U_0}\inf\; \log\bb{\abs{\phi_2(z)}}} =\tau_1^\alpha\cdot C(U),
	\end{align*}
	where $C(U)$ is a constant which depends on the domain $U$ (and the choice of $\tilde r$) alone, and $\alpha<1$ therefore we can always choose $\tau_1$ large enough so that $\tau_1>C(U)\cdot{\tau_1}^\alpha$, making $v_1$ a subharmonic function in $\C$.
	
	Let $u_1\colon \mathbb{C}\to \R$ be the subharmonic function obtained by applying the puncture lemma, Lemma \ref{lem:punctures}, with the disks $B\bb{-\tau_1,\frac{\tau_1}4}$, and  $B\bb{\tau_1,\frac{\tau_1}4}$, and the underling map $v_1$ constructed above. 
	
	We shall now see that the map $u_1$ satisfies the properties in the statement.  Choosing $\tau_1>\frac{4}{3}\bb{A_U+1}$, the disks $B\bb{-\tau_1,\frac{\tau_1}4}$, $B(0, A_U+1)$ and $B\bb{\tau_1,\frac{\tau_1}4}$ are pairwise disjoint, and by definition, $u_1(z)=v_1(z)=\tau_1\bb{\abs z-A_U}^{\alpha}$ for $z$ outside these domains, i.e., \eqref{item:prop_sh_def} holds. Next, since $\overline{U_0}\subset B(0, A_U+1)$, $u_1\equiv v_1$ in that domain. Moreover, on $\overline{U_0}\setminus V_1$, $u_1(z)=v_1(z)>0$ by \eqref{eq:U_1}, i.e.,  property \eqref{item:prop_sh_bound2} is satisfied.
	
	Note that property \eqref{item:prop_sh_bound1} already holds on $\overline{U}_1$ by the way the constant $C_1$ was chosen and as $\left.u_1\right|_{\overline{U_1}}\equiv\left.v_1\right|_{\overline{U_1}}$. To see that it holds on the disks $B\bb{-\tau_1,\frac{\tau_1}4}$ and $B\bb{\tau_1,\frac{\tau_1}4}$, note that whenever $\abs z\ge\frac{3\tau_1}4,$
	\begin{align*}
		\Delta v_1&=\tau_1\bb{\alpha^2-\alpha}\bb{\abs z-A_U}^{\alpha-2}+\frac1{\abs z}\cdot\alpha \tau_1\bb{\abs z-A_U}^{\alpha-1}=\frac{\tau_1\alpha^2}{\bb{\abs z-A_U}^{2-\alpha}}+\tau_1\cdot\alpha\bb{\abs z-A_U}^{\alpha-1}\bb{\frac1{\abs z}-\frac1{\abs z-A_U}}\\
		&=\frac{\tau_1\alpha^2}{\bb{\abs z-A_U}^{2-\alpha}}-\frac{\tau_1\cdot\alpha\cdot A_U}{\abs z\bb{\abs z-A_U}^{2-\alpha}}=\frac{\tau_1\alpha^2}{\bb{\abs z-A_U}^{2-\alpha}}\bb{1-\frac{A_U}{\alpha\abs z}}\ge \frac{\tau_1\alpha^2}{2\bb{\abs z-A_U}^{2-\alpha}}
	\end{align*}
	assuming $\tau_1>\frac{8A_U}{3\alpha}$ to guarantee that the last inequality holds for $\abs z\ge\frac{3\tau_1}4.$
	
	Using this and part \eqref{item:punctures} or Lemma \ref{lem:punctures} with $\delta=\frac{16}{\tau_1}$, $\tau_1>16$, and $\pm\tau_1$ standing for either $-\tau_1$ or $\tau_1$,
	\begin{align*}
		\underset{z\in B\bb{\pm\tau_1,4}}\max\;u_1(z)&\le \underset{z\in B\bb{\pm\tau_1,\frac{\tau_1}4}}\max\; v_1(z)-c\cdot\bb{\frac {\tau_1}{4}}^2\cdot \bb{\underset{z\in  B\bb{\pm\tau_1,\frac{\tau_1}4}}\inf\; \Delta u_1}\cdot \log\bb{\frac{\tau_1}{16}}\\
		&\le \tau_1\bb{\bb{\tau_1+\frac{\tau_1}4-A_U}^\alpha-c\cdot \bb{\frac {\tau_1}{4}}^2\cdot \frac{\alpha^2}{2\bb{\tau_1+\frac {\tau_1}4-A_U}^{2-\alpha}}\log\bb{\frac{\tau_1}{16}}},\\
		&\le \tau_1\bb{\frac54}^\alpha  \tau_1^\alpha\bb{\bb{1-\frac{4A_U}{5\tau_1}}^\alpha- \frac{c\cdot\alpha^2}{32}\bb{\frac54}^{-2} \bb{1-\frac{4A_U}{5\tau_1}}^{\alpha-2}\log\bb{\frac{\tau_1}{16}}}\le-3\tau_1^\alpha,
	\end{align*}
	for $\tau_1$ chosen to be large enough, depending on the fixed constants $c$, $\alpha$, and $A_U$, and using that $\tau_1>1$, concluding the proof of \eqref{item:prop_sh_bound1}, and of the proposition.
\end{proof}

\paragraph*{Bounding the integral}~\\
Following the definition of the annuli $\mathcal{A}_1$ and $\mathcal{A}_2$ in \eqref{eq_annuli_f1} and Proposition~\ref{prop:chi}, let $\chi_1\colon \C\rightarrow[0,1]$ be a smooth function such that
\begin{enumerate}
	\item If $z \in B(-\tau_1,\frac{\tau_1}4)\cup V_1\cup B(\tau_1,\frac{\tau_1}4) $, then $\chi_1(z)=1$.
	\item If $z \notin B(-\tau_1,\frac{\tau_1}3)\cup U_0\cup B(\tau_1,\frac{\tau_1}3)$, then $\chi_1(z)=0$.
	\item 	$$
	\abs{\nabla\chi_1(z)}\le 
	\begin{cases}
		\frac D{\tau_1},& z\in \mathcal{A}_1\cup \mathcal{A}_2,\\
		D,& z\in U_0\setminus V_1,\\
		0,& \text{otherwise},
	\end{cases}
	$$
	for some constant $D>1$ that depends on $\dist(\partial U_0, \partial V_1)$, and on $C_U$, the bi-Lipschitz constant of the map~$\varphi_U$.
\end{enumerate}

Note that if $z\in \mathcal{A}_1\cup \mathcal{A}_2$, then, by part \eqref{item:prop_sh_def} of Proposition \ref{prop:sh_u1},
$$u_1(z)= \tau_1\bb{\abs z-A_U}^{\alpha}\geq \tau_1\bb{\frac{\tau_1}{4}-A_U}^{\alpha}\geq \tau_1$$
if $\tau_1$ is large enough. In addition, following condition \eqref{cond:2}, $\tau_2^{11}\le\exp\bb{\tau_1^\alpha}$ implying that
$$
\tau_2^2\cdot e^{-\tau_1^\alpha}\le e^{-\tau_1^\alpha\bb{1-\frac2{11}}}<1.
$$
Using this, and the fact that the area  of $\mathcal{A}_1\cup \mathcal{A}_2$ is proportional to $\tau_1^2$
\begin{align*}
	\iint_\C\abs{\bar\partial\chi_1(z)\cdot h(z)}^2e^{-u_1(z)}dm(z)&=\iint_{\mathcal{A}_1\cup \mathcal{A}_2\cup U_0\setminus V_1}\abs{\bar\partial\chi_1(z)}^2\abs{h(z)}^2e^{-u_1(z)}dm(z)\\
	&\le \frac {D^2}{\tau_1^2}\bb{\iint_{\mathcal{A}_1}\abs{-\tau_1}^2e^{-\tau_1}dz+\iint_{\mathcal{A}_2}\bb{\frac{4\tau_1}3+\tau_2}^2e^{-\tau_1}dz}+D^2\iint_{U_0\setminus V_1}\bb{\tau_1+2}^2e^{-u_1(z)}dz\\
	&\le D^2\left(2\tau_2^2e^{-\tau_1}+\tau_1^2\right)<D_1^2\cdot\tau_1^2:=\mathcal I^2.
\end{align*}

Following Proposition \ref{prop:method}, let $f_1$ be the entire function defined by $f_1(z)=\chi_1(z)\cdot h_1(z)-\beta_1(z)$, where $\beta_1$ is H\"ormander's solution to the $\bar\partial$-equation $\bar\partial\beta_1=\bar\partial\chi_1\cdot h_1$. 

\paragraph*{Estimating the growth (property \ref{itm:growth}):} Let $z\ge 2\tau_1$. Note that following Proposition \ref{prop:sh_u1}, if $\abs{z-w}<1$ then
\begin{align*}
	\abs{u_1(w)-u_1(z)}&=\tau_1\abs{\bb{\abs w-A_U}^\alpha-\bb{\abs z-A_U}^\alpha}\le \tau_1\bb{\bb{\abs z+1-A_U}^\alpha-\bb{\abs z-A_U}^\alpha}\\
	&=\tau_1\abs z^\alpha\bb{1-\frac{A_U}{\abs z}}^\alpha\bb{\bb{1+\frac1{\abs z-A_U}}^\alpha-1}\le 2\tau_1\cdot\abs z^{-(1-\alpha)}
\end{align*}
since by Bernoulli's inequality
$$
\bb{1-\frac{A_U}{\abs z}}^\alpha\bb{\bb{1+\frac1{\abs z-A_U}}^\alpha-1}\le 1\cdot\bb{1+\frac\alpha{\abs z-A_U}-1}<\frac2{\abs z}.
$$
Combining this with part \ref{item:growth_bound} of Proposition~\ref{prop:method},
$$
\abs{f_1(z)}^2e^{-u_1(z)}\le\bb{\frac{2+2\abs z^2}{\sqrt \pi}}^2\cdot D_1^2\cdot\tau_1^2 \cdot \underset{w\in B(z,1)}\max\; e^{u_1(w)-u_1(z)}\le \frac{4}{\pi}\abs z^4\bb{1+\frac1{\abs z^2}}^2\cdot D_1^2\cdot\tau_1^2\cdot 2\tau_1\cdot\abs z^{-(1-\alpha)}<15\cdot \abs z^4\tau_1^6
$$
if $\tau_1$ is large enough concluding the proof of property \ref{itm:growth} for $f_1$. 

\paragraph*{Iterations (property \ref{itm:iterations}):}
Let $z\in \overline{U_2}$, and let $r:=\frac12\dist\bb{\partial U_2,\partial V_1}$, then $B(z,r)\subset V_1$ and $\chi_1|_{B(z,r)}\equiv 1$. Then, by part \ref{item:error_bound} of Proposition~\ref{prop:method} and part \eqref{item:prop_sh_bound1} of Proposition \ref{prop:sh_u1}, we have

\begin{equation}\label{eq:first_c}
	\abs{f_1(z)-\bb{\varphi_U\inv(z)+\tau_{1}}}=\abs{f_1(z)-h_1(z)}\le \frac{2+2\cdot \diam(U_0)^2}{r }\cdot D_1\cdot\tau_1\cdot e^{-3\tau_1^\alpha}\le e^{-\tau_1^\alpha}<\frac{1}{\tau_2},
\end{equation}
if $\tau_1$ is large enough and using \eqref{cond:2}, which proves \ref{itm:iterations}\ref{subitm:first}.

Additionally, let $w\in B\bb{\tau_1,3}$, then $B(w,1)\subset B\bb{\tau_1,4}$ and $\chi_1|_{B(w,1)}\equiv 1$. Arguing similarly,
\begin{equation}\label{eq:second_c}
	\abs{f_1(w)-\bb{w+\bb{\tau_2-\tau_1}}}=\abs{f_1(w)-h_1(w)}\le 2\bb{1+\bb{\tau_1+3}^2}\cdot D_1\cdot\tau_{1}\cdot e^{-3\tau_1^\alpha}\le e^{-\tau_1^\alpha}<\frac{1}{\tau_2},
\end{equation} 
taking $\tau_1$ to be large enough. Next, let $z\in\overline{U_2}$. Combining \eqref{eq:first_c} and \eqref{eq:second_c} with $w=f_1(z)$ we have
\[
\abs{f_1^{2}(z)-\big(\varphi_U^{-1}(z)+\tau_2\big)}
\le \abs{f_1^{2}(z)-\big(f_1(z)+(\tau_2-\tau_1)\big)}
+\abs{f_1(z)-\big(\varphi_U^{-1}(z)+\tau_1\big)}\\
< \frac{2}{\tau_2}
< \frac{9}{\tau_2}.
\]
This concludes the proof of property \ref{itm:iterations}.

\paragraph*{Derivatives (property \ref{itm:derivative})}
Using Cauchy's integral formula for derivatives, see \eqref{eq:derivative}, if $z\in\overline{U_2}$, then $B(z,r)\subset U_2^{+r}$ where $r=\frac12\dist\bb{\partial U_2,\partial V_1}$, and so, 
$$
\abs{f'_1(z)-\left(\varphi_U\inv\right)'}=\abs{f_1'(z)-h_1'(z)}\le\frac1{r}\cdot \underset{z\in U_2^{+r}}\sup\;\abs{h_1(z)-f_1(z)}<\frac{1}{\tau_2},
$$
following \eqref{eq:first_c}. Combining this with \eqref{eq:derivative_phiU}, the first item of \ref{itm:derivative} follows. 

Similarly, since for $z\in\overline{U_2}$ we know that $B(f_1(z), 1)\subset B(\tau_1,3)$, for any $w\in B(f_1(z), 1)$,
$$
\abs{f_1'(w)-1}=\abs{f_1'(w)-h_1'(w)}\le \underset{\zeta\in B(\tau_1,3)}\sup\;\abs{h_1(\zeta)-f_1(\zeta)} \le e^{-\tau_1^\alpha}<\frac{1}{\tau_2},
$$
following \eqref{eq:second_c} and concluding the proof of \ref{itm:derivative}.

\paragraph*{Attracting basin (property \ref{itm:attracting}):}
Let $z\in B\bb{-\tau_1,1}$, then $B(z,1)\subset B\bb{-\tau_1,3}$ and $\chi_1|_{B(z,1)}\equiv 1$. Then, by part \ref{item:error_bound} of Proposition~\ref{prop:method} and part \eqref{item:prop_sh_bound1} of Proposition \ref{prop:sh_u1},
$$
\abs{f_1(z)-\bb{-\tau_1}}=\abs{f_1(z)-h_1(z)}\le 2\bb{1+\bb{\tau_1+2}^2}\cdot D_1\cdot \tau_{1}\cdot e^{-3\tau_1^\alpha}\le e^{-\tau_1^\alpha}<\frac{1}{\tau_2},
$$
if $\tau_1$ is large enough, proving \ref{itm:attracting} and concluding the proof of Theorem \ref{thm:sequence} for $k=1$.
\subsection{Proof of Theorem \ref{thm:sequence} for $k\ge 2$:}
We now assume that there exist functions $\bset{f_j}_{j=1}^{k-1}$ for which the theorem holds and we construct $f_k$.

\paragraph*{The model map}~\\
Roughly speaking, we enforce the following behaviour: On the disk, \(B\!\left(0,\tfrac{10}{36}\tau_k\right)\), which contains the first \(k\) forward images of
\(U_{k+1}\) under \(f_{k-1}\), the map \(f_k\) closely matches \(f_{k-1}\).
By the inductive hypothesis, \ref{itm:iterations}, \(f_{k-1}^{\,k}(V_{k+1})\) is close to a round disk centred at
\(\tau_k\). Our aim is to define \(f_k\) so that it maps \(\overline{U}_{k+1}\subset V_{k+1}\) to a set
that is \emph{almost} a disk centred at \(\tau_{k+1}\), while simultaneously sending neighbourhoods of the
points \(f_{k-1}^{\,k}(a_j^k)\) into a neighbourhood of \(\bb{-\tau_1}\).
This is achieved using the local construction done in Lemma~\ref{lem:small_disks}  (see Section \ref{sec:local}), and is
formalised in the next lemma. 

In order to simplify the notation, for the collection $Q_k=\{a_j^k\}^{N_k}_{j=0}\subset \partial U_k$ of points fixed in \eqref{eq:points_Qk}, define the collection of points 
\begin{equation}\label{eq:points_bk}
	\left\{b^k_j=\frac{f^k_{k-1}(a^k_j)-\tau_k}{r_k} ,\;\; a^k_j\in Q_k \right\}.
\end{equation}

\begin{prop}\label{prop:gk}There exists an entire map $g_k\colon \C\to \C$ satisfying
	\begin{enumerate}[label=($\roman*_k$),leftmargin=0.7cm]
		\item\label{item:att_basin_k} For every $\delta\in\bb{0,1}$ and for every point $b_j^k,\; r_k \cdot g_k\bb{B\bb{b_j^k, \frac{1}{50}\delta\cdot s\cdot\rho_k}}\subset B\bb{-\tau_1,\frac{\delta^2}{\tau_{k+1}^2}}$;
		\item\label{item:err_disk} For every $z\in\overline{U_{k+1}}$,
		\begin{enumerate}
			\item \label{subitm:riemann_k} 
			$\abs{r_k\cdot g_k\bb{\frac{f_{k-1}^k(z)-\tau_k}{r_k}}-\bb{\varphi_U\inv(z)+\tau_{k+1}}}<\frac1{\tau_{k+1}}.$
			\item  \label{subitm:derivative_k} if  $w\in B\bb{f_{k-1}^k(z),\frac{\rho_k}{10}}$, then   $\abs{\abs{g_k'\bb{\frac{w-\tau_k}{r_k}}}-1}<\frac{\log^3\bb{\tau_k}}{2\tau_{k}}$.
		\end{enumerate}
		\item\label{item:growth_k} $M_{g_k}\bb{\frac{\tau_{k}}3}<\frac{\exp\bb{\log^7\bb{\tau_{k}}}}{r_k}.$
	\end{enumerate}
\end{prop}
\begin{proof}
	Define the domain  $$W_k=\frac{f_{k-1}^{k}(V_{k+1})-\tau_k}{r_k}$$ and the map $\varphi_k\colon \D\to W_k$ by
	$$
	\varphi_k(z):=\frac1{r_k}\bb{f_{k-1}^k\bb{\varphi_U\bb{r_k\cdot z}}-\tau_k}.
	$$
	
	First, we claim that the function $\varphi_k$ is conformal and is a Riemann map from $\D$ onto $W_k$. To see it is onto, observe that $\varphi_U\bb{\D_{r_k}}=V_{k+1}$, implying that $\varphi_k(\D)=W_k$. To see it is conformal, note that following the induction hypothesis, 
	$$
	\abs{f_{k-1}^k(z)-\bb{\varphi_U\inv(z)+\tau_k}}<\frac{3^k}{\tau_k}.
	$$
	Following Proposition \ref{prop:tau1} part \ref{item:tau1_e} we conclude it is conformal.
	
	%
	
	For every $z\in \D$ we can write $z=\frac1{r_k}\varphi_U\inv\bb{\varphi_U\bb{r_k\cdot z}}$, and by inclusion $\varphi_U\bb{r_k\cdot z}\in V_{k+1}\subset U_k$ implying that
	\begin{align*}
		\dist(\partial\D,\partial W_k)
		&=\inf_{|\zeta|=1}\dist\bigl(\zeta,\partial W_k\bigr) 
		\le \sup_{w\in\partial V_{k+1}}
		\left|\,\frac{f_{k-1}^k(w)-\tau_k}{r_k}
		-\frac{\varphi_U^{-1}(w)}{r_k}\right| \\
		&\le \frac{1}{r_k}\;
		\sup_{w\in\partial V_{k+1}}
		\bigl|\,f_{k-1}^k(w)-\bigl(\varphi_U^{-1}(w)+\tau_k\bigr)\,\bigr|
		\;<\;\frac{1}{r_k}\cdot \frac{3^k}{\tau_k}.
	\end{align*}
	following part \ref{subitm:final} of property \ref{itm:iterations} of the function $f_{k-1}$.
	
	Similarly, the points $a_j^k$ are in $\partial U_k$, using property \ref{itm:iterations}\ref{subitm:final} of the function $f_{k-1}$ again, we have
	\begin{equation}\label{eq_distance_atrr}
		\begin{aligned}
			\abs{f_{k-1}^k\bb{a_j^k}-\tau_k}&\ge \abs{\varphi_U\inv\bb{a_j^k}}-\abs{f_{k-1}^k\bb{a_j^k}-\bb{\varphi_U\inv\bb{a_j^k}+\tau_k}}>\bb{1+\rho_k}-\frac{3^k}{\tau_k},\\
			\abs{f_{k-1}^k\bb{a_j^k}-\tau_k}&\le \abs{\varphi_U\inv\bb{a_j^k}}+\abs{f_{k-1}^k\bb{a_j^k}-\bb{\varphi_U\inv\bb{a_j^k}+\tau_k}}<\bb{1+\rho_k}+\frac{3^k}{\tau_k},
		\end{aligned}
	\end{equation}
	which implies that $\bset{b_j^k}\subset\frac{1}{r_k}\cdot \A\bb{1+\rho_k-\frac{3^k}{\tau_k},1+\rho_k+\frac{3^k}{\tau_k}}.$ Define  $\gamma_k=\frac{1+\rho_k}{r_k}-1= \frac{\rho_k-\rho_{k+1}}{2+\rho_k+\rho_{k+1}}$. We will show that
	\begin{equation}\label{eq:gammak}
		\frac{1}{5}\rho_k\leq \gamma_k\leq \frac{1}{2}\rho_k
	\end{equation}
	if $\tau_1$ is chosen large enough: the upper bound is immediate noting that \(\rho_{k+1}\le\rho_k\) and \(2+\rho_k+\rho_{k+1}\ge2\). For the lower bound, note that since $\rho_k\searrow 0$, we may choose $\tau_1$ so that $\rho_k<\rho_1<\frac 1{4}$. Using condition \eqref{cond:1} on the sequence $\bset{\tau_k}$ we see that
	\begin{align*}
		\frac{\rho_{k+1}}{2\rho_k}=\frac{\log\log(\tau_{k+1})}{2\log\log(\tau_{k+2})}\le\frac14
		\quad\Longrightarrow\quad\gamma_k= \frac{\rho_k-\rho_{k+1}}{2+\rho_k+\rho_{k+1}}=\rho_k\cdot \frac{\frac12-\frac{\rho_{k+1}}{2\rho_k}}{1+\frac{\rho_k}2+\frac{\rho_{k+1}}2}\ge \rho_k\cdot\frac{\frac14}{1+\frac1{4}}=\frac{\rho_k}5.
	\end{align*}
	In particular, 
	$$
	\bset{b_j^k}\subset \A\bb{1+\gamma_k-\frac{3^k}{r_k\tau_k},1+\gamma_k+\frac{3^k}{r_k\tau_k}}.
	$$
	Using hypothesis \ref{itm:iterations}\ref{subitm:final} for every $j\neq \ell$ we see that
	\begin{align}\label{eq:bk_separation}
		\abs{b_j^k-b^k_{\ell}}&=\frac{1}{r_k}\abs{f_{k-1}^k\bb{a_j^k}-f_{k-1}^k\bb{a_\ell^k}}\\
		&\ge \frac{1}{r_k}\bb{\abs{\varphi_U\inv\bb{a_j^k}-\varphi_U\inv(a_\ell^k)}-\abs{f_{k-1}^{k}(a_j^k)-\bb{\varphi_U\inv(a_j^k)+\tau_{k}}}-\abs{f_{k-1}^{k}(a_\ell^k)-\bb{\varphi_U\inv(a_\ell^k)+\tau_{k}}}}\nonumber\\
		&\ge \frac{1}{r_k}\abs{\varphi_U\inv\bb{a_j^k}-\varphi_U\inv(a_\ell^k)}-2\cdot \frac{3^k}{r_k\tau_k}>C_U\inv\cdot \rho_k-2\cdot \frac{3^k}{r_k\tau_k}\ge\frac{C_U\inv}2\rho_k\ge\frac{\rho_k}4\ge\frac{\gamma_k}2\nonumber
	\end{align}
	as $\varphi_U$ is bi-Lipschitz with the assumption $\frac12\le C_U\inv$ made in \eqref{eq:derivative_phiU}, and the points $a^k_j$ are $(r_k\cdot \rho_k)$-separated (as $\frac{\rho_k}4\ge \frac{3^k}{r_k\cdot \rho_k}$ following Proposition \ref{prop:tau1} part \ref{item:tau1_b}).
	
	Apply Lemma~\ref{lem:small_disks} with the conformal map $\varphi_k:\D\to W_k$, the centers $\{b_j^k\}$, and the choices of parameters
	\[
	\kappa=\tfrac16,\quad
	\eta=\gamma_k,\quad
	A=-\tfrac{\tau_1}{r_k},\quad
	\tau=\tfrac{\tau_{k+1}}{r_k},\quad
	\varphi=\varphi_k:\D\to W_k,\quad
	\varepsilon=\tfrac{3^k}{r_k\tau_k},\quad
	\{b_j\}=\{b_j^k\}.
	\]
	Note that following Proposition \ref{prop:tau1} part \ref{item:tau1_b}
	$$
	3\eps\log\bb{\frac1\eps}=3\tfrac{3^k}{r_k\tau_k}\log\bb{\tfrac{r_k\tau_k}{3^k}}\le\frac{3^{k+1}}{\tau_k}\log\bb{\tau_k}<\rho_k.
	$$
	In addition, using \eqref{eq:bk_separation}, \eqref{eq:gammak}, and \eqref{eq_distance_atrr} (enlarging $\tau_1$ if needed), the separation, smallness, and size hypotheses of the lemma are satisfied. Note that with this choice of parameters, 
	\begin{align}\label{eq:err_loc}
		&E_1:=\frac{D\cdot\tau}{\eta^{\frac32}\sqrt\kappa}=\frac{D\cdot\tfrac{\tau_{k+1}}{r_k}}{\gamma_k^{\frac32}\sqrt{\tfrac16}}\le\frac{6D\tau_{k+1}}{\rho_k^{\frac32}};\quad\quad\quad\quad
		E_2:=\exp\bb{-\frac{\kappa}2\cdot e^{\frac1\eta-\frac\kappa2}}=\exp\bb{-\frac1{12}\cdot e^{\frac1{\gamma_k}-\frac1{12}}}\le \exp\bb{-\frac{e^{\frac2{\rho_k}}}{20}}\nonumber\\
		&\Rightarrow E_1\cdot E_2^{\frac1{40}}\le \frac{6D\tau_{k+1}}{\rho_k^{\frac32}}\cdot \exp\bb{-\frac{e^{2\log\log\tau_{k+1}}}{800}}\le \tau_{k+1}^2\exp\bb{-\log^{\frac32}(\tau_{k+1)}}\le\frac1{36\tau_{k+1}^2},
	\end{align}
	since $\gamma_k\le\frac{\rho_k}2$, if we choose $\tau_1$ numerically large enough.
	
	Let $g_k$ be the entire function that the Lemma \ref{lem:small_disks} provides. Let us check that it satisfies the properties in the statement using the properties of $g_k$ stated in Lemma \ref{lem:small_disks}:
	
	To see \ref{item:att_basin_k}, let $\delta\in\bb{0,1}$. Then, using \eqref{eq:gammak}, for every $j$ and every $z\in B\bb{b_j^k, \delta\cdot \frac{1}{50}s\cdot\rho_k}$,
	$$
	\abs{r_k\cdot g_k(z)-\tau_1}=r_k\cdot \abs{g_k(z)-A}<r_k\cdot E_1\cdot E_2\cdot \delta^{\frac 1{D}e^{\frac1{\gamma_k}}-1}\le\frac{\delta^2}{\tau_{k+1}^2},
	$$
	choosing $\tau_1$ large enough so that $\frac 1{D}e^{\frac1{\gamma_1}}>3$, as desired.	
	
	The second property in the statement of Lemma \ref{lem:small_disks} holds for points such that $\vert z \vert <1-\frac{1}{6}\gamma_k$. In order to prove~\ref{item:err_disk},  first we show that for any $z\in\overline{U_{k+1}}$ and $w\in B\bb{f_{k-1}^k(z),\frac{\rho_k}{10}}$, using hypothesis \ref{itm:iterations}\ref{subitm:final} and  \eqref{eq:gammak}, 
	\begin{align*}
		\abs{\frac{w-\tau_k}{r_k}}&\le \abs{f_{k-1}^k\bb{z}-w}+\abs{\frac{f_{k-1}^k\bb{z}-\tau_k}{r_k}}\le \frac{\rho_k}{10}+\frac{\abs{\varphi_U\inv(z)}+\abs{f_{k-1}^k\bb{z}-\bb{\varphi_U\inv(z)+\tau_k}}}{r_k}\\
		&\le \frac{\rho_k}{10}+\frac{1+\rho_{k+1}+\frac{3^k}{\tau_k}}{r_k}=1+\frac{\rho_k}{10}-\frac{r_k-\bb{1+\rho_{k+1}}-\frac{3^k}{\tau_k}}{r_k}=1+\frac{\rho_k}{10}-\gamma_k+\frac{3^k}{r_k\tau_k}<1-\frac{\gamma_k}{2}+\frac{\gamma_k}{3}
		\ =1-\frac16\gamma_k,
	\end{align*}
	since $\frac{3^k}{r_k\tau_k}\le \frac{\rho_k-\rho_{k+1}}{6}\leq \frac{1}{3}\gamma_k$.
	
	To prove part (a) of \ref{item:err_disk}, note that for $z\in\overline{U_{k+1}}\subset V_{k+1}$,  $\frac{f_{k-1}^k(z)-\tau_k}{r_k}\in W_k$ and $ \varphi_U\inv(z)=r_k\cdot\varphi_k\inv\bb{\frac{f_{k-1}^k(z)-\tau_k}{r_k}}$. Then, applying the second property of $g_k$ in Lemma \ref{lem:small_disks}, and \eqref{eq:err_loc}, we have
	\begin{align*}
		&\abs{r_k\cdot g_k\bb{\frac{f_{k-1}^k(z)-\tau_k}{r_k}}-\bb{\varphi_U\inv(z)+\tau_{k+1}}}=\\
		&=\abs{r_k\cdot g_k\bb{\frac{f_{k-1}^k(z)-\tau_k}{r_k}}-\bb{r_k\cdot \varphi_k\inv\bb{\frac{f_{k-1}^k(z)-\tau_k}{r_k}}+\tau_{k+1}}}<r_k\cdot \frac1{2\tau_{k+1}^2}<\frac{1}{\tau_{k+1}^2}.
	\end{align*} 
	
	To see part (b) of \ref{item:err_disk}, fix  $z\in\overline{U_{k+1}}$, and let $w\in B\bb{f_{k-1}^k(z),\frac{\rho_k}{10}}$. 
	Then
	\begin{align*}
		\abs{\abs{g_k'\bb{\frac{w-\tau_k}{r_k}}}-1}\le\abs{g_k'\bb{\frac{w-\tau_k}{r_k}}-\frac{\varphi_k'(0)}{\abs{\varphi_k'(0)}}}&\le \frac{36E_1\cdot E_2^{\frac1{40}}}{\gamma_k}+\frac{72}{\gamma_k}\cdot\tfrac{3^k}{r_k\tau_k}\log\bb{\tfrac{r_k\tau_k}{3^k}}\le \frac{\log^3\bb{\tau_k}}{2\tau_{k}},
	\end{align*}
	since as we saw in the proof of part \ref{item:tau1_b} in proposition \ref{prop:tau1}, $\log\log(\tau_{k+1})\le \log(\tau_{k})$ implying that
	$$
	\frac{36E_1\cdot E_2^{\frac1{40}}}{\gamma_k}+\frac{72}{\gamma_k}\cdot\tfrac{3^k}{r_k\tau_k}\log\bb{\tfrac{r_k\tau_k}{3^k}}\le 100\tfrac{3^k\cdot \log\log(\tau_{k+1})}{\tau_k}\log\bb{\tfrac{\tau_k}{3^k}}\le 100\tfrac{3^k\cdot \log^2(\tau_{k})}{\tau_k}\le \frac{\log^3(\tau_{k})}{2\tau_k}.
	$$
	
	Finally, to see \ref{item:growth_k}, by the last conclusion of Lemma \ref{lem:small_disks},
	\begin{align*}
		M_{g_k}\bb{\frac{\tau_{k}}3}&\le \tau_{k+1}+ \bb{\frac{\tau_{k}}3}^{\frac D{\gamma_k} e^{\frac1{\gamma_k}}}\leq \tau_{k+1}+\exp\bb{\frac {5D}{\rho_k} e^{\frac5{\rho_k}}\log\bb{\frac{\tau_{k}}3}}\\
		&\leq \tau_{k+1}+\exp\bb{\log^6\bb{\tau_{k+1}}\log\log\bb{\tau_{k+1}}}<\frac{\exp\bb{\log^7\bb{\tau_{k+1}}}}{r_k},
	\end{align*}
	since $\gamma_k\ge\frac{\rho_k}5$, if $\tau_1$ is numerically large enough.
\end{proof}	

We now define the model map as
\begin{equation}\label{def:hk} 
	h_k(z)=	\begin{cases}
		f_{k-1}(z),& z\in B\bb{0,\frac{10}{36}\tau_k},\\
		r_k\cdot g_k\bb{\frac{z-\tau_k}{r_k}},& z\in B\bb{\tau_k,\frac{10}{36}\tau_k},\\
		0,& \text{otherwise}.
	\end{cases}
\end{equation}

\begin{figure}[h]
	\centering
	\def\svgwidth{\textwidth}
	\begingroup%
	\makeatletter%
	\providecommand\color[2][]{%
		\errmessage{(Inkscape) Color is used for the text in Inkscape, but the package 'color.sty' is not loaded}%
		\renewcommand\color[2][]{}%
	}%
	\providecommand\transparent[1]{%
		\errmessage{(Inkscape) Transparency is used (non-zero) for the text in Inkscape, but the package 'transparent.sty' is not loaded}%
		\renewcommand\transparent[1]{}%
	}%
	\providecommand\rotatebox[2]{#2}%
	\newcommand*\fsize{\dimexpr\f@size pt\relax}%
	\newcommand*\lineheight[1]{\fontsize{\fsize}{#1\fsize}\selectfont}%
	\ifx\svgwidth\undefined%
	\setlength{\unitlength}{5102.36220472bp}%
	\ifx\svgscale\undefined%
	\relax%
	\else%
	\setlength{\unitlength}{\unitlength * \real{\svgscale}}%
	\fi%
	\else%
	\setlength{\unitlength}{\svgwidth}%
	\fi%
	\global\let\svgwidth\undefined%
	\global\let\svgscale\undefined%
	\makeatother%
	\begin{picture}(1,0.27777778)%
		\lineheight{1}%
		\setlength\tabcolsep{0pt}%
		\put(0,0){\includegraphics[width=\unitlength,page=1]{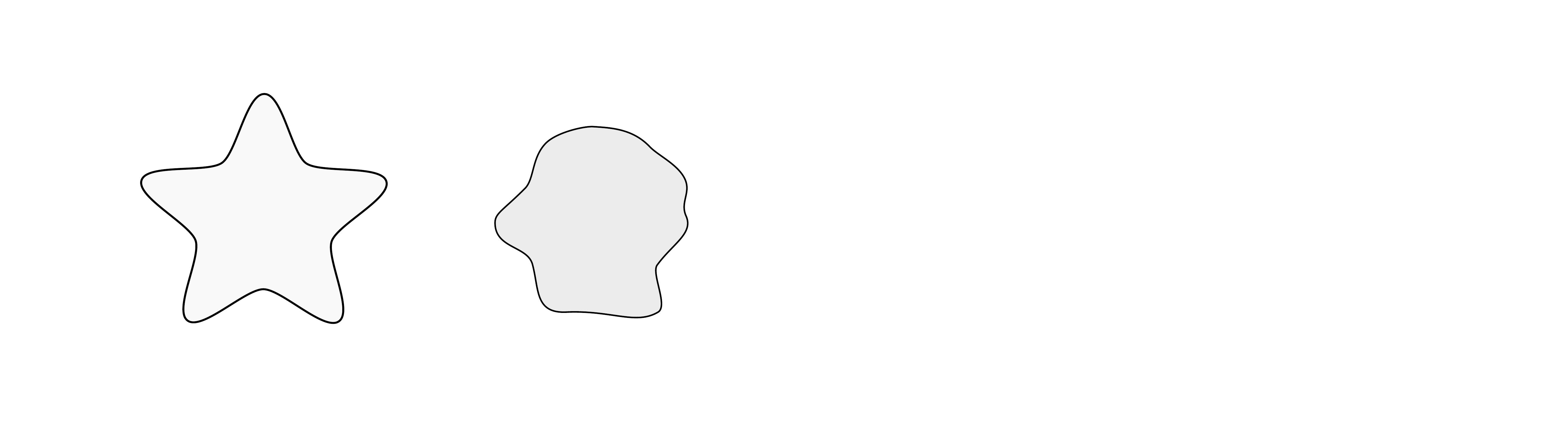}}%
		\put(0.76743421,0.12453912){\color[rgb]{0,0,0}\makebox(0,0)[lt]{\lineheight{1.25}\smash{\begin{tabular}[t]{l}$\fontsize{9pt}{1em}\mathcal{A}_4$\end{tabular}}}}%
		\put(0,0){\includegraphics[width=\unitlength,page=2]{induction.pdf}}%
		\put(0.7623061,0.12107904){\color[rgb]{0,0,0}\makebox(0,0)[lt]{\lineheight{1.25}\smash{\begin{tabular}[t]{l}$\fontsize{8pt}{1em}\tau_k$\end{tabular}}}}%
		\put(0.93572898,0.1223022){\color[rgb]{0,0,0}\makebox(0,0)[lt]{\lineheight{1.25}\smash{\begin{tabular}[t]{l}$\fontsize{8pt}{1em}\tau_{k+1}$\end{tabular}}}}%
		\put(0.18700163,0.20013508){\color[rgb]{0,0,0}\makebox(0,0)[lt]{\lineheight{1.25}\smash{\begin{tabular}[t]{l}$\fontsize{8pt}{1em}U_0$\end{tabular}}}}%
		\put(0.54019845,0.01919509){\color[rgb]{0,0,0}\makebox(0,0)[lt]{\lineheight{1.25}\smash{\begin{tabular}[t]{l}$\fontsize{8pt}{1em}B(0, \frac{10}{36}\tau_{k})$\end{tabular}}}}%
		\put(0.52004772,0.19441439){\color[rgb]{0,0,0}\makebox(0,0)[lt]{\lineheight{1.25}\smash{\begin{tabular}[t]{l}$\fontsize{8pt}{1em}f_{k-1}^{k-1}(U_{k})$\end{tabular}}}}%
		\put(0.35997365,0.2686093){\color[rgb]{0,0,0}\makebox(0,0)[lt]{\lineheight{1.25}\smash{\begin{tabular}[t]{l}$\fontsize{8pt}{1em}f_k$\end{tabular}}}}%
	\end{picture}%
	\endgroup%
	
	\caption{Schematic of the inductive step of the construction (not to scale). Inside the disk \(B\!\left(0,\tfrac{10}{36}\tau_k\right)\), the map \(f_k\) agrees closely with \(f_{k-1}\), which sends the first \(k-1\) forward images of \(U_{k+1}\) into a set close to a round disk centred at \(\tau_k\). The model map \(h_k\), that \(f_k\) approximates, sends \(\overline{U}_{k+1}\) to a disk centred at \(\tau_{k+1}\), while disks centred at the points \(f^{\,k}_{k-1}(a_j^k)\) are mapped into a neighbourhood of \((-\tau_1)\). In this step we approximate these two behaviours by implementing the two maps just described near \(0\) and near \(\tau_k\), respectively using the puncture lemma to obtain precise control near their centres (where the relevant sets lie), and then gluing them with a cut-off function, \(\chi\), for which \(\nabla\chi\) is supported in thin annuli around the corresponding disks.}\label{fig:induction}
	
\end{figure}

\paragraph*{The subharmonic function}~\\
Let $u_k\colon \mathbb{C}\to \R$ be the subharmonic function obtained by applying the puncture lemma, Lemma \ref{lem:punctures}, with the disks $ B\bb{0,\frac{\tau_k}{4}}$, $B\bb{\tau_k,\frac{\tau_k}{4}}$, and the underlying subharmonic function 
$$v\colon \mathbb{C}\to \R; \quad  v(z)=\abs z^\alpha.$$
Observe that the function $\abs z\mapsto \Delta v(z)=\frac{\alpha\abs z^\alpha}{\abs z^2}$ 
is monotone decreasing in $\abs z$, and therefore, using Lemma \ref{lem:punctures} part \eqref{item:punctures} with $\delta=\frac{16}{\tau_k}<1$,
\begin{align}\label{eq:bnd_pst}
	\underset{z\in B\bb{\tau_k,4}}\max\;u_k(z)&\le \underset{z\in B\bb{\tau_k,\frac{\tau_k}4}}\max\; u_k(z)-c\cdot \frac{\tau_k^2}{16}\cdot \bb{\underset{z\in  B\bb{\tau_k,\frac{\tau_k}4}}\inf\; \Delta u_k}\cdot \log\bb{\frac{\tau_k}{16}}\\
	&\le \bb{2\tau_k}^\alpha-c\cdot\frac{\tau_k^2}{16}\cdot\alpha\cdot \bb{2\tau_k}^{\alpha-2}\log\bb{\frac{\tau_k}{16}}\leq \bb{2\tau_k}^\alpha\bb{1-\frac{c\cdot\alpha}{16}\log\bb{\frac{\tau_k}{16}}}<-3\tau_k^\alpha,\nonumber
\end{align}
for $\tau_1$ large enough, depending on the numerical constant  $c>0$ and on $\alpha$. 

Similarly, for $\delta=\frac{8\tau_{k-1}+8}{\tau_k}$, we have
\begin{align}\label{eq:bnd_pst_2}
	\underset{z\in B\bb{0,2\tau_{k-1}+2}}\max\;u_k(z)&\le \underset{z\in B\bb{0,\frac{\tau_k}4}}\max\; u_k(z)-c\cdot \frac{\tau_k^2}{16}\cdot \bb{\underset{z\in  B\bb{0,\frac{\tau_k}4}}\inf\; \Delta u_k}\cdot \log\bb{\frac{\tau_k}{8\tau_{k-1}+8}}\\
	&\le \bb{2\tau_k}^\alpha\bb{1-\frac{c\cdot\alpha}{16}\log\bb{\frac{\tau_k}{8\tau_{k-1}+8}}}<-3\tau_k^\alpha,\nonumber
\end{align}
if $\tau_1$ is chosen to be large enough for the last inequality to hold.

\paragraph*{Bounding the integral}~\\
Denote the two `double' annuli 
$$ \mathcal{A}_1:=  \A\bb{\frac14\tau_k,\frac13\tau_k}\setminus \A\bb{\frac{10}{36}\tau_k,\frac{11}{36}\tau_k}, \quad \quad \text{ and } \quad \quad  \mathcal{A}_2=\tau_k+ \mathcal{A}_1.$$ 
By Proposition \ref{prop:chi}, we can construct a smooth map $\chi_k:\C\rightarrow[0,1]$ such that
\begin{enumerate}
	\item If $z\notin  \A\bb{\frac14\tau_k,\frac13\tau_k}\cup  \left( \tau_k +\A\bb{\frac14\tau_k,\frac13\tau_k} \right)$, then $\chi_k(z)=1$. 
	\item If $z\in \A\bb{\frac{10}{36}\tau_k,\frac{11}{36}\tau_k}\cup \left( \tau_k +\A\bb{\frac{10}{36}\tau_k,\frac{11}{36}\tau_k}\right)$, then $\chi_k(z)=0$.
	\item If $z\in \mathcal{A}_1\cup \mathcal{A}_2$, 
	then $\abs{\nabla\chi_k(z)}\leq \frac{D_0}{\tau_k}$ for some numerical constant $D_0>0$. 
\end{enumerate}
By the definition of the map $h_k$ in \eqref{def:hk}, 
$$
\abs{h_k(z)}\le	\begin{cases}
	\abs{f_{k-1}(z)},& z\in\mathcal A_1,\\
	r_k\cdot \abs{g_k\bb{\frac{z-\tau_k}{r_k}}},& z\in\mathcal A_2.
\end{cases}
$$
In addition, $\left.u_k\right|_{\mathcal A_1\cup\mathcal A_2}\equiv \abs z^{\alpha}$, and the area of $\mathcal A_1\cup\mathcal A_2$ is proportional to $\tau_k^2$. Then, using this, the bounds from hypothesis \ref{itm:growth}, and part \ref{item:growth_k} of Proposition \ref{prop:gk}, 
\begin{align*}
	\iint_\C\abs{\bar\partial\chi_k(z)\cdot h_k(z)}^2e^{-u_k(z)}dm(z)&=\iint_{\mathcal{A}_1\cup \mathcal{A}_2}\abs{\bar\partial\chi_k(z)}^2\abs{h_k(z)}^2e^{-u_k(z)}dm(z)\\
	&\le \frac{D_0^2}{\tau^2_k}\cdot m(\mathcal{A}_1\cup \mathcal{A}_2)\cdot \bb{ \underset{z\in \mathcal{A}_1}\max\; \vert f_{k-1}(z)\vert ^2 e^{-\abs z^\alpha}+  \underset{z\in \mathcal{A}_2}\max\; \abs{ r_k\cdot g_{k}\bb{\frac{z-\tau_k}{r_k}}}^2 e^{-\abs z^\alpha}}\\
	&\le D_1 \left(\tau_k^4\cdot\tau_{k-1}^6+ e^{2\log^7(\tau_{k+1})-\left(\frac{2\tau_k}3\right)^\alpha} \right)\le \tau_k^6=:\mathcal{I}^2,
\end{align*}
for some numerical constant $D_1>0$, following \eqref{eq:exp_growth}, and Condition \eqref{cond:2}, as
\begin{equation}\label{eq:bound8}
	2\log^7(\tau_{k+1})\le 2\bb{\frac1{11}\tau_{k}^{\frac\alpha7}}^7<\frac1{50}\tau_k^\alpha<\left(\frac{2\,\tau_k}{3}\right)^{\alpha}.
\end{equation}


Following Proposition \ref{prop:method}, let $f_k$ be the entire function defined by $f_k(z)=\chi_k(z)\cdot h_k(z)-\beta_k(z)$, where $\beta_k$ is H\"ormander's solution to the $\bar\partial$-equation $\bar\partial\beta_k=\bar\partial\chi_k\cdot h_k$. 


\paragraph*{Estimating the errors (property \ref{itm:consistency}):}

First, let $z\in B\bb{0,2\tau_{k-1}+1}$, then $B\bb{z,1}\subset B\bb{0,2\tau_{k-1}+2}$ and $\chi_k|_{B(z,1)}\equiv 1$. Following part \ref{item:error_bound} of Proposition \ref{prop:method}, and the bounds described in \eqref{eq:bnd_pst_2} on $u_k|_{B(z,1)}$,
$$
\abs{f_k(z)-f_{k-1}(z)}=\abs{f_k(z)-h_k(z)}\le 2\bb{1+\bb{2\tau_{k-1}+2}^2}\cdot \tau_{k}^3\cdot e^{-\frac32\cdot \tau_k^\alpha}\le \tau_k^6\cdot e^{-\frac12\tau_k^\alpha-\tau_k^\alpha}\le e^{-\tau_k^\alpha}<\frac{1}{\tau_{k+1}},
$$
where the second-to-last inequality holds if $\tau_1$ is chosen sufficiently large, and the last one follows from \eqref{cond:2}. Item  \ref{itm:consistency} in the statement of the theorem follows.

Next, we will show that the approximation is close to the model map in $B(\tau_k,2)$. While this will not be used right-away, we will need this to prove item \ref{itm:iterations}, and the technique is similar to the proof above. Every $z\in B\bb{\tau_k,2}$, satisfies $\chi_k|_{B(z,1)}\equiv 1$ since, $B\bb{z,1}\subset B\bb{\tau_k,4}$. Arguing as above and using the bound \eqref{eq:bnd_pst} on $u_k|_{B(z,1)}$, 
\begin{equation}\label{eq:error_tauk}
	\abs{f_k(z)-h_k(z)}\le 2\bb{1+\bb{\tau_k+4}^2}\cdot \tau_k^3\cdot e^{-\frac32\cdot \tau_k^\alpha}\le\tau_k^6\cdot e^{-\frac12\tau_k^\alpha-\tau_k^\alpha}\le e^{-\tau_k^\alpha}.
\end{equation}	


\paragraph*{Estimating the growth (property \ref{itm:growth}):}  Let $\abs z\ge 2\tau_k$. Applying part \ref{item:growth_bound} of Proposition~\ref{prop:method}, and noting that if $\abs z\ge 2\tau_k$, then the model map, $h_k$, is zero in $B(z,1)$ yields 
$$
\abs{f_k(z)}^2 e^{-\abs z^\alpha}\le  2\bb{\frac{2+2\abs z^2}{\sqrt\pi}}^2\cdot \tau_k^6\cdot \underset{w\in B(z,1)}\max\; e^{u_k(w)-\abs z^\alpha} \le\frac {8\cdot e}\pi\abs z^4\bb{1+\frac1{\abs z^2}}^2\cdot \tau_k^6\\
<15\abs z^2\cdot\tau_k^6
$$
as
\begin{align}\label{eq:exp_diff}
	\underset{w\in B(z,1)}\max\; e^{u_k(w)-\abs z^\alpha}\le \exp\bb{\abs z^\alpha\bb{\bb{1+\frac1{\abs z}}^\alpha-1}}\le \exp\bb{\alpha\cdot \abs z^{-\bb{1-\alpha}}}<e,
\end{align}
since by Bernoulli’s inequality $\bb{1+t}^\alpha\le 1+\alpha\cdot t$ for $\alpha\in\bb{0,1}$, applied with $t=\frac{1}{\abs z}$, concluding the proof of \ref{itm:growth}.

\paragraph*{Iterations and derivatives (properties \ref{itm:iterations} and \ref{itm:derivative}):}

The first part of \ref{itm:derivative}, namely the lower and upper bounds on the modulus of $f'_k$, is a direct consequence of the same property for $f_{k-1}$, which holds by the inductive hypothesis, \eqref{eq:derivative}, and property \ref{itm:consistency} that we already proved:
$$
\left.\abs{f'_k}\right|_{\overline{U_2}}<\left.\abs{f'_{k-1}}\right|_{\overline{U_2}}+\frac{\underset{\abs{z-w}=1}\max\abs{f_k-f_{k-1}}}1<C_U+\sumit j 1 {k-1}\frac{1}{\tau_{j+1}}+\frac{1}{\tau_{k+1}}=C_U+\sumit j 1 k\frac{1}{\tau_{j+1}},
$$
and similarly, 
\begin{equation}\label{eq:deriv_lower_U2}
	\left.\abs{f'_k}\right|_{\overline{U_2}}>\left.\abs{f'_{k-1}}\right|_{\overline{U_2}}-\frac{\underset{\abs{z-w}=1}\max\abs{f_k-f_{k-1}}}1>C^{-1}_U-\sumit j 1 {k-1}\frac{1}{\tau_{j+1}}-\frac{1}{\tau_{k+1}}>C^{-1}_U-\sumit j 1 k\frac{1}{\tau_{j+1}}>C^{-1}_U-\frac{2}{\tau_2}.
\end{equation}

Our next goal is to show the second item of property \ref{itm:iterations} and the remaining estimate of \ref{itm:derivative}. We will do so intertwined and by induction on the iterates, $\nu$.  More precisely, we will prove the following auxiliary result:
\begin{prop}
	For every $2\leq \nu\le k$,
	\begin{enumerate}  
		\renewcommand{\labelenumi}{(\arabic{enumi})}
		\item\label{item:deriv2} For every $\nu\le m\le k$, for every $w\in\partial U_m, z\in\overline{U_{m+1}}$, 
		$$
		\abs{f_{k-1}^{\nu-1}(z)-f_{k-1}^{\nu-1}(w)}\ge\prodit \ell 1 {\nu-2}\bb{1-\frac{2\log^2\bb{\tau_\ell}}{\tau_\ell}}\left(C_U^{-1}-\frac{2}{\tau_2}\right) \cdot C_U\inv\bb{\rho_{m}-\rho_{m+1}}\geq c_0 \frac{3^m}{\tau_m}>\frac{3^{m}}{\tau_{m+1}}$$
		for some universal constant $c_0>0$, if $\tau_1$ is chosen large enough. 
		\item \label{item:deriv1} \ref{itm:iterations}\ref{subitm:iterations} holds, namely, for every $z\in\overline{U_{\nu+1}}$ and $\nu-1 \le j\le k-1$, $\abs{f_k^\nu(z)-f_j^\nu(z)}<3^\nu\sumit m {j+1}{k}\frac{1}{\tau_{m+1}}<\frac{2\cdot 3^{\nu}}{\tau_{j+2}}$.
		\item\label{item:deriv3} \ref{itm:derivative}\ref{subitem:derivative_iterations} holds, namely, if $w\in\underset{z\in \overline{U_{\nu+1}}}\bigcup B\bb{f_k^{\nu}(z),\frac{\rho_{\nu}}{10}-\sumit\ell\nu k\frac{3^\ell}{\tau_{\ell+1}}}$, for some $z\in \overline{U_{\nu+1}}$ and $\nu\leq k$, then $\abs{f'_k\bb{w}-1}<\frac{\log^3\bb{\tau_\nu}}{2\tau_{\nu}}+\sumit \ell {\nu} k\frac{1}{\tau_{\ell+1}}<\frac{\log^3(\tau_\nu)}{\tau_\nu}$.
	\end{enumerate}
\end{prop}
Note that the first statement, \eqref{item:deriv2}, is an auxiliary claim that will be required to prove the other two.

\begin{proof}
	{\bf The base case $\nu=2$:} To see \eqref{item:deriv2}, first note that, as we saw in the proof of \eqref{eq:gammak}, 
	$$
	\rho_m-\rho_{m+1}=\rho_m\bb{1-\frac{\rho_{m+1}}{\rho_m}}\ge \frac{\rho_m}2>\frac{3^m}{\tau_m},
	$$
	following part \ref{item:tau1_b} of Proposition \ref{prop:tau1}. Next, note that if $\tau_1\ge 10$, 
	$$
	\prod_{\ell \geq 2} \bb{1-\frac{2\log^3\bb{\tau_\ell}}{\tau_\ell}}=\exp\bb{\sumit \ell 2\infty \log\bb{1-\frac{2\log^3\bb{\tau_\ell}}{\tau_\ell}}}\ge \exp\bb{-\sumit \ell 2\infty \frac{2\log^3\bb{\tau_\ell}}{\tau_\ell}}\ge\exp\bb{-\sumit \ell 2\infty \frac{2\ell^3}{10^\ell}}\ge e^{-\frac14}>\frac34.
	$$
	Define $c_0=\frac{1}{40}\left(C_U^{-1}-\frac{2}{\tau_2}\right) \cdot C_U\inv$ and 
	choose $\tau_1$ large enough so that 
	$c_0\,\frac{3^{m}}{\tau_m} \;>\; \frac{3^{m}}{\tau_{m+1}}$ for every $m\geq 2$.
	Using these estimates and \eqref{eq:deriv_lower_U2}, for every $2\le m\le k$ and every $z\in \overline{U_{m+1}}$ and $w\in \partial U_{m}$,
	\begin{equation}\label{eq:bound_aux}
		\begin{aligned}
			\abs{f_{k-1}(z)-f_{k-1}(w)}&\ge \underset{z\in \overline{U_{2}}}\inf\abs{f'_{k-1}(z)}\abs{z-w}\ge \left(C_U^{-1}-\frac{2}{\tau_2}\right) \cdot C_U\inv\bb{\bb{1+\rho_m}-\bb{1+\rho_{m+1}}}\\
			&=\left(C_U^{-1}-\frac{2}{\tau_2}\right) \cdot C_U\inv\bb{\rho_{m}-\rho_{m+1}}>c_0\frac{3^m}{\tau_{m}}>\frac{3^{m}}{\tau_{m+1}}>\frac{1}{\tau_{k+1}}, 
		\end{aligned}
	\end{equation}
	which concludes the proof of \eqref{item:deriv2} in the case $\nu=2$. In addition, using property \ref{itm:consistency}, which we already proved, and \eqref{eq:bound_aux} for  $m=2$, we get that 
	$$
	f_k\bb{\overline{U_3}}\subset \underset{z\in \overline{U_3}}\bigcup B\bb{f_{k-1}(z),\frac1{\tau_{k+1}}}\subset f_{k-1}\bb{\overline{U_{2}}}.
	$$
	To prove \eqref{item:deriv1}, fix $z\in\overline{U_3}$, and, for now,  $j=k-1$. Since $z, f_k(z), f_{k-1}(z)\in B\bb{0,2\tau_1}\subset B\bb{0,2\tau_{k-1}}$, using \ref{itm:consistency}, and the inclusion above,
	\begin{equation}\label{eq:squared}
		\begin{aligned}
			\abs{f_k^2(z)-f_{k-1}^2(z)}&<\abs{f_k\bb{f_k(z)}-f_{k-1}\bb{f_k(z)}}+\abs{f_{k-1}\bb{f_{k}(z)}-f_{k-1}\bb{f_{k-1}(z)}}\\
			&<\frac1{\tau_{k+1}}+\underset{w\in B\bb{f_{k-1}(z),\frac{1}{\tau_{k+1}}}}\sup\;\abs {f'_{k-1}(w)}\abs {f_k(z)-f_{k-1}(z)}\\
			&<\frac1{\tau_{k+1}}\bb{1+\underset{f_{k-1}\bb{\overline{U_{2}}}}\sup\;\abs {f'_{k-1}}}<\frac1{\tau_{k+1}}\bb{2+\frac{\log^3\bb{\tau_2}}{\tau_{2}}}<\frac{3}{\tau_{k+1}},
		\end{aligned}
	\end{equation}
	where in the last line we bounded the derivative of $f_{k-1}$ using part \ref{subitem:derivative_iterations} of \ref{itm:derivative}.
	
	Next, for any $1\le j\le k-2$, we combine this with part \ref{subitm:iterations} of \ref{itm:iterations} for $f_{k-1}$:
	$$
	\abs{f_k^2(z)-f_j^2(z)}\le \abs{f_k^2(z)-f_{k-1}^2(z)}+\abs{f_{k-1}^2(z)-f_j^2(z)}<\frac{3}{\tau_{k+1}}+3\sumit m j {k-1}\frac1{\tau_{m+2}}=3\sumit m {j+1}{k}\frac{1}{\tau_{m+1}}.
	$$
	
	To see item \eqref{item:deriv3} for $\nu=2$, let $w\in B\bb{f^2_k(z),\frac{\rho_2}{10}-\sumit m 2 k\frac{3^m}{\tau_{m+1}}}$ for some $z\in\overline{U_3}$. Then, arguing similarly by using \eqref{eq:derivative}, \eqref{eq:squared}, and part \ref{subitem:derivative_iterations} of \ref{itm:derivative} for $f_{k-1}$,
	\begin{align*}
		\abs{f'_k\bb{w}-1}&\le \abs{\bb{f_k-f_{k-1}}'\bb{w}}+\underset{\zeta\in B\bb{f^2_{k-1}(z),\frac{\rho_2}{10}-\sumit m 2 k\frac{3^m}{\tau_{m+1}}+\frac1{\tau_{k+1}}}}\max\abs{f'_{k-1}\bb{\zeta}-1}\\
		&<\frac1{\tau_{k+1}}+\underset{\zeta\in \bb{f^2_{k-1}\bb{\overline{U_{2}}}}^{+\bb{\frac{\rho_2}{10}-\sumit m 2 {k-1}\frac{3^m}{\tau_{m+1}}}}}\sup \abs{f'_{k-1}\bb{\zeta}-1}<\frac{\log^3\bb{\tau_2}}{2\tau_{2}}+\sumit \ell {2} k\frac{1}{\tau_{\ell+1}}.
	\end{align*}
	This concludes the proof of the base case of the induction.
	
	{\bf The induction step:} Assume \eqref{item:deriv2}-\eqref{item:deriv3} hold for $(\nu-1)$, and we shall prove it for $\nu$.
	To see \eqref{item:deriv2}, let $\nu\leq m\leq k$. Let $z\in\overline{U_{m+1}}, w\in\partial U_m$. Then, using part \ref{subitem:derivative_iterations} of \ref{itm:derivative} for $f_{k-1}$ and $\nu-2$, and the inductive hypothesis, i.e.,  that \eqref{item:deriv2} holds for $\nu-1$, we have
	\begin{align*}
		\abs{f_{k-1}^{\nu-1}(z)-f_{k-1}^{\nu-1}(w)}&=\abs{f_{k-1}\bb{f_{k-1}^{\nu-2}(z)}-f_{k-1}\bb{f_{k-1}^{\nu-2}(w)}}\ge \underset{\zeta \in  f_{k-1}^{\nu-2}\bb{\overline{U_{\nu-1}}}}\inf \abs{f'_{k-1}(\zeta)} \abs{f_{k-1}^{\nu-2}(z)-f_{k-1}^{\nu-2}(w)}\\
		&\ge \bb{1-\frac{2\log^3{\tau_{\nu-2}}}{\tau_{\nu-2}}}\abs{f_{k-1}^{\nu-2}(z)-f_{k-1}^{\nu-2}(w)}\ge\prodit \ell 1 {\nu-2}\bb{1-\frac{2\log^3\bb{\tau_\ell}}{\tau_\ell}}C_U^{-2}\bb{\rho_m-\rho_{m+1}}>\frac{3^{m}}{\tau_{m+1}}.
	\end{align*}
	Moreover, using the inductive hypothesis, i.e., \eqref{item:deriv1} for $\nu-1$ and $j=k-1$, and the inequality we just proved for $m=\nu$,
	$$
	f_k^{\nu-1}\bb{\overline{U_{\nu+1}}}\subset \underset{z\in \overline{U_{\nu+1}}}\bigcup B\bb{f_{k-1}^{\nu-1}(z),\frac{2\cdot 3^{\nu-1}}{\tau_{k+1}}}\subset f_{k-1}^{\nu-1}\bb{\overline{U_{\nu}}}.
	$$
	To prove \eqref{item:deriv1}, fix $z\in\overline{U_{\nu+1}}$. As in the base case, we start by showing the case $j=k-1$. Since $f_k^{\nu-1}(\overline{U_{\nu+1})}\subset B(0,2\tau_{k-1})$, applying the consistency  property, property \ref{itm:consistency}, and part \ref{subitm:iterations} of property \ref{itm:iterations} for $f^{\nu-1}_{k-1}$, combined with the inclusion above we see that
	\begin{equation}\label{eq:iterates2}
		\begin{aligned}
			\abs{f_k^\nu(z)-f_{k-1}^\nu(z)}&<\abs{f_k\bb{f_k^{\nu-1}(z)}-f_{k-1}\bb{f_k^{\nu-1}(z)}}+\abs{f_{k-1}\bb{f_k^{\nu-1}(z)}-f_{k-1}\bb{f_{k-1}^{\nu-1}(z)}}\\
			&<\frac1{\tau_{k+1}}+\underset{w\in B\bb{f_{k-1}^{\nu-1}(z),\frac{2\cdot 3^{\nu-1}}{\tau_{k+1}}}}\sup\;\abs {f'_{k-1}(w)}\abs {f_k^{\nu-1}(z)-f_{k-1}^{\nu-1}(z)}\\
			&<\frac1{\tau_{k+1}}+\underset{f_{k-1}^{\nu-1}\bb{\overline{U_\nu}}}\sup\;\abs {f'_{k-1}}\cdot\frac{2\cdot 3^{\nu-1}}{\tau_{k+1}}\le\frac1{\tau_{k+1}}+ \bb{1+\frac{\log^3\bb{\tau_{\nu-1}}}{\tau_{\nu-1}}}\cdot\frac{2\cdot 3^{\nu-1}}{\tau_{k+1}}\leq 
			\frac{3^\nu}{\tau_{k+1}},
		\end{aligned}
	\end{equation}
	where the derivative of $f_{k-1}$ on $U_{\nu-1}$ is bounded by  part \ref{subitem:derivative_iterations} of property \ref{itm:derivative}. Next, for any $1\le j\le k-2$,  combining this with part \ref{subitm:iterations} of property \ref{itm:iterations} for $f_{k-1}$,
	$$
	\abs{f_k^\nu(z)-f_j^\nu(z)}\le \abs{f_k^\nu(z)-f_{k-1}^\nu(z)}+\abs{f_{k-1}^\nu(z)-f_j^\nu(z)}\le \frac{3^\nu}{\tau_{k+1}}+3^\nu\sumit m {j+1} {k-1}\frac1{\tau_{m+1}}=3^\nu\sumit m {j+1}{k}\frac{1}{\tau_{m+1}}.
	$$
	Lastly, to deduce \eqref{item:deriv3}, we divide the proof into two cases. First, suppose that $\nu<k$. Then, arguing as before, using \eqref{eq:derivative}, \eqref{eq:iterates2}, and part \ref{subitem:derivative_iterations} of property \ref{itm:derivative} for $f^{\nu}_{k-1}$, whenever $w\in B\bb{f_k^{\nu}(z),\frac{\rho_{\nu}}{10}-\sumit\ell\nu k\frac{3^\ell}{\tau_{\ell+1}}}$ for some $z\in\overline{U_{\nu+1}}$,
	\begin{align*}
		\abs{f'_k(w)-1}&\le \abs{\bb{f_k-f_{k-1}}'(w)}+\underset{\zeta\in B\bb{f_{k-1}^{\nu}(z),\frac{\rho_\nu}{10}-\sumit\ell\nu k\frac{3^\ell}{\tau_{\ell+1}}+\frac{3^{\nu}}{\tau_{k+1}}}}\max\abs{f'_{k-1}\bb{\zeta}-1}\\
		&<\frac1{\tau_{k+1}}+\underset{\zeta\in \bb{f_{k-1}^{\nu}\bb{\overline{U_{\nu+1}}}}^{+\bb{\frac{\rho_\nu}{10}-\sumit\ell\nu {k-1}\frac{3^\ell}{\tau_{\ell+1}}}}}\sup\abs{f'_{k-1}\bb{\zeta}-1}<\frac{\log^2\bb{\tau_\nu}}{2\tau_{\nu}}+\sumit \ell {\nu} k\frac{1}{\tau_{\ell+1}}.
	\end{align*}
	If $\nu=k$, item \eqref{item:deriv1} for $j=k-1$ implies that for every $z\in\overline{U_{k+1}}$,
	$f_k^k(z)\in B\bb{f_{k-1}^k(z),\frac{2\cdot 3^k}{\tau_{k+1}}}.$ 
	Combining this and part \ref{subitm:final} of property \ref{itm:iterations} for the map $f^k_{k-1}$, $f_k^k(\overline{U}_{k+1})\subset B(\tau_k,2)$ means we can then use \eqref{eq:error_tauk}.
	
	Moreover, by part (b) of property \ref{item:err_disk} in Proposition~\ref{prop:gk},  every $\zeta\in B\bb{f_{k-1}^k(z),\frac{\rho_k}{10}}$ satisfies $\abs{h_k'(\zeta)-1}<\frac{\log^3{\tau_k}}{2\tau_k}$, implying that for every $w\in B\bb{f_{k}^k(z),\frac{\rho_k}{10}-\frac{2\cdot 3^k}{\tau_{k+1}}}$,
	$$
	\abs{f'_k(w)-1}\le \abs{\bb{f_k-h_k}'(w)}+\underset{\zeta\in B\bb{f_{k-1}^{k}(z),\frac{\rho_k}{10}}}\sup\abs{h'_k\bb{\zeta}-1}<e^{-\tau_k^\alpha}+ \frac{\log^3\bb{\tau_k}}{2\tau_{k}} \leq \frac{\log^3\bb{\tau_k}}{\tau_{k}},
	$$
	concluding the proof of the proposition.
\end{proof}
Note that the proposition concludes the proof of item \ref{itm:derivative} and of parts \ref{subitm:first} and \ref{subitm:iterations} in \ref{itm:iterations}.

We are left to prove part \ref{subitm:final} of property \ref{itm:iterations}. Let $z\in\overline{U_{k+1}}$. Using  \eqref{eq:error_tauk},  that $f_k^k(z)\in B\bb{f_{k-1}^k(z),\frac{2\cdot 3^k}{\tau_{k+1}}}$, and property \ref{item:err_disk} 
of the function $g_k$ from Proposition \ref{prop:gk}, we have
\begin{align*}
	\abs{f_k^{k+1}(z)-\bb{\varphi_U\inv(z)+\tau_{k+1}}}&\le \abs{f_k\bb{f_k^k(z)}-h_k\bb{f_k^k(z)}}+\abs{h_k\bb{f_{k}^k(z)}-\bb{\varphi_U\inv(z)+\tau_{k+1}}}\\
	&\le e^{-\tau_k^\alpha}+\abs{h_k\bb{f_k^k(z)}-h_k\bb{f_{k-1}^k(z)}}+\abs{h_k\bb{f_{k-1}^k(z)}-\bb{\varphi_U\inv(z)+\tau_{k+1}}}\\
	&\le e^{-\tau_k^\alpha}+\underset{\zeta\in B\bb{f_{k-1}^{k}(z),\frac{2\cdot 3^{k}}{\tau_{k+2}}}}\sup\abs{h'_k\bb{\zeta}}\abs{f_{k-1}^k(z)-f_{k}^k(z)}+\frac1{\tau_{k+1}}\\
	&\le e^{-\tau_k^\alpha}+\left( 1+ \frac{\log^3\bb{\tau_k}}{2\tau_{k}}\right)\frac{2\cdot 3^k}{\tau_{k+1}}+\frac1{\tau_{k+1}}
	<\frac{3^{k+1}}{\tau_{k+1}},
\end{align*}
concluding the proof of property \ref{itm:iterations}.

\paragraph*{Attracting basin (property \ref{itm:attracting}):} 
By definition of the map $h_k$ in \eqref{def:hk} and part \ref{item:att_basin_k} of Proposition \ref{prop:gk}, for every $\delta\in\bb{0,1}$ for every $w\in B\bb{f^k_{k-1}(a^k_j), \frac{1}{50}\delta\cdot s\cdot\rho_k}\subset B\bb{f^k_{k-1}(a^k_j), \frac{r_k}{50}\cdot \delta\cdot s\cdot\rho_k}$, 
$$h_k(w)=r_k\cdot g_k\bb{\frac{w-\tau_k}{r_k}} \subset B\bb{-\tau_1,\frac{\delta^2}{\tau_{k+1}}}.$$ 

By this, choosing $\delta=\frac{1}{4}$, noting that $\rho_k<\frac{s}{200}$ for $\tau_1$ large enough, and using \eqref{eq:error_tauk}, 
$$
f_k \bb{B\bb{f^k_{k-1}(a^k_j),\rho_k^2}}\subset  f_k \bb{B\bb{f^k_{k-1}(a^k_j), \frac{s}{200} \cdot \rho_k}}\subset B\bb{-\tau_1,\frac{1}{\tau_{k+1}}+e^{-\tau_k^\alpha}}\subset B\bb{-\tau_1,\frac{1}{2}},
$$
concluding the proof of Theorem \ref{thm:sequence}.

\section{The limiting function: the proof of Theorem 1.1}
	Let $\bset{f_k}$ be the sequence of entire functions provided by Theorem~\ref{thm:sequence} for some fixed domain, $U$, a value of $\alpha$ as in the statement of the theorem, and some sequence $\bset{\tau_k}$ satisfying \eqref{cond:2} and \eqref{cond:1}. Note that by property \ref{itm:consistency} of the sequence $\bset{f_k}$ and as $\bset{\tau_k}\nearrow\infty$,
for every $j$ fixed sequence, $\bset{f_k}$, is a Cauchy sequence when restricted to $B(0, 2\tau_{j})$. We conclude that the sequence $\bset{f_k}$ converges locally uniformly to an entire function, $f$, as $\bigcup^\infty_{k=1}B(0, 2\tau_{k})=\C$. Moreover, for every $k\geq 2$ and $z\in B(0, 2\tau_{k-1})$, 
\begin{equation}\label{eq_error_f}
	\vert f(z)-f_{k-1}(z)\vert 	\le \sumit \ell {k} {\infty}\abs{f_{\ell}(w)-f_{\ell-1}(w)}<\sumit \ell {k} {\infty}\frac{1}{\tau_{\ell+1}}\leq \frac{2}{\tau_{k+1}}.
\end{equation}

Let $A:=B\bb{-\tau_1,\frac{3}{4}}$ and let $Q_k=\bset{a_j^k}_{j=1}^{N_k}$ be the sequence of collections of points fixed in \eqref{eq:points_Qk}, whose accumulation set as $k\to \infty$ is $\partial U$. We will show that 
\begin{equation}\label{eq:attracting}
	f(\overline{A}) \cup f^{k+1}(Q_k)\subset A \quad \text{ for all } k\geq 1.
\end{equation}
Indeed, by property \ref{itm:attracting} of the sequence $\bset{f_k}$, $f_1(\overline{A})\subset B\bb{-\tau_1, \frac{1}{\tau_1}}$, and since, by \eqref{eq_error_f}, $\vert f(z)-f_1(z)\vert< \frac{2}{\tau_1}$, the inclusion $f(\overline{A})\subset A$ follows. Let us fix some $k\geq 1$ and some point  $a_j^k\in Q_k$ for some $j$. 
Note that since $\bset{f_m}$ converges locally uniformly to $f$, then the sequence $\bset{f_m^k}$ converges locally uniformly to $f^k$, where the rate of convergence depends on $k$. Let $m\geq k$ be large enough so that $\underset{z\in B(0,2\tau_k)}\sup\;\abs{f_m^k-f^k}<\frac{\rho_k^2}2$. Then, by part \ref{subitm:iterations} of property \ref{itm:iterations} applied to $f_m$ with $\nu=j=k-1$,
$$
\abs{f_m^{k-1}\bb{a_j^k}-f^{k-1}_{k-1}\bb{a_j^k}}\le \frac{3^{k-1}}{\tau_{k+1}},
$$
as $a_j^k\in\overline{U_k}$. 	Moreover, properties \ref{itm:consistency} and part \ref{subitm:final} of property \ref{itm:iterations}, imply that $f_m^{k-1}\bb{a_j^k}\in B\bb{0,2\tau_{k-1}}$ therefore
$$
\abs{f_m\bb{f_m^{k-1}\bb{a_j^k}}-f_{k-1}\bb{f_m^{k-1}\bb{a_j^k}}}<\sumit \ell k m\frac1{\tau_{\ell+1}}<\frac2{\tau_{k+1}}.
$$
Combining this with property \ref{itm:derivative} for $f^{k-1}_{k-1}$, we see that
\begin{align*}
	\abs{f_m^k\bb{a_j^k}-f^k_{k-1}\bb{a_j^k}}	&=\abs{f_m\bb{f_m^{k-1}\bb{a_j^k}}-f_{k-1}\bb{f_{k-1}^{k-1}\bb{a_j^k}}}\\
	&	\le \abs{f_m\bb{f_m^{k-1}\bb{a_j^k}}-f_{k-1}\bb{f_m^{k-1}\bb{a_j^k}}}+\abs{f_{k-1}\bb{f_m^{k-1}\bb{a_j^k}}-f_{k-1}\bb{f_{k-1}^{k-1}\bb{a_j^k}}}\\
	&\le  \frac2{\tau_{k+1}}+\underset{w\in B\bb{f^{k-1}_{k-1}\bb{a_j^k},\frac{3^{k-1}}{\tau_{k}}}}\sup\; \abs{f_{k-1}'(w)}\abs{f_m^{k-1}\bb{a_j^k}-f^{k-1}_{k-1}\bb{a_j^k}}\le\frac2{\tau_{k+1}}+2\cdot  \frac{3^{k-1}}{\tau_{k+1}}\\
	&<\frac{4}{3}\cdot\frac{3^k}{\tau_{k+1}}
	\le \frac{4}{3}\cdot\frac{3^k}{\tau_k}\cdot\frac{\tau_k}{\tau_{k+1}}
	\le \frac{\rho_k^2}{2},
\end{align*}
\M{since} $\ \frac{3^k\log^3(\tau_k)}{\tau_k}\le \rho_k$, the last bound holds for $\tau_1$ large enough; see Proposition \ref{prop:tau1}.

Overall, we have that
$$
\abs{f^k\bb{a_j^k}-f^k_{k-1}\bb{a_j^k}}\le \abs{f^k\bb{a_j^k}-f^k_m\bb{a_j^k}}+\abs{f_m^k\bb{a_j^k}-f^k_{k-1}\bb{a_j^k}}\le \underset{z\in B(0,2\tau_k)}\sup\;\abs{f_m^k-f^k}+\frac{\rho_k^2}2<\rho_k^2.
$$
i.e., $f^k(a_j^k)\in B\bb{f_{k-1}^k(a_j^k),\rho_k^2}$.
This together with  property of the sequence $\bset{f_k}$ and \eqref{eq_error_f} yields
$$f^{k+1}(a_j^k)\in f\bb{B\bb{f^k_{k-1}\bb{a_j^k}, \rho_k^2}}\subset B\bb{-\tau_1,\frac1{2}+\frac{2}{\tau_{k+1}}}\subset A,$$
which concludes the proof of \eqref{eq:attracting}.

Equation \eqref{eq:attracting} together with Montel's theorem imply that all the points in $A \cup \bigcup_{k\geq 1} Q_k$ are contained in attracting basins of $f$, and, in particular, they must have bounded orbits.  On the other hand, property \ref{itm:iterations} of the sequence $\bset{f_k}$, combined with \eqref{eq_error_f} imply that that $f^k(U)\subset B(\tau_k, 3)$ for each $k$, and since $\tau_k\to \infty$, all points in $U$ have unbounded orbits. Consequently, any neighbourhood of any $z\in \partial U$ contains both points with bounded and unbounded orbits, which prevents the family of iterates to be normal on it. We conclude that $\partial U \subset J(f)$. Since $\tau_k>0$ for all $k$, $\bigcup_k f^k(\overline{U})$ is contained in some right half plane, and since $U$ is regular, applying Montel's theorem we obtain that $\text{int}(\overline{U})=U\subset F(f)$. Overall, as $U$ is bounded, connected, and its iterates converge to infinity, it must be a wandering Fatou component.

The fact that the iterates $f^n\vert_{U}$ are univalent follows from part \ref{subitm:final} of property \ref{itm:iterations} of the sequence $\bset{f_k}$ together with the choice of $\tau_1$ in Proposition \ref{prop:tau1}.

Finally, fix $R>2\tau_1$ and let $k\in\N$ be so that $2\tau_{k-1}< R\le2 \tau_{k}$. Then, using property \ref{itm:growth} of the sequence $\bset{f_k}$, combined with \eqref{eq_error_f} and condition \eqref{cond:2}, we see that
\begin{align}\label{eq:growth-ratef}
	\left(M_f(R)\right)^2e^{-R^\alpha}&\le  2 \underset{\abs z=2\tau_k}\max\;\abs{f_{k}(z)}^2e^{-\abs z^\alpha}+2\underset{\abs w\le 2\tau_{k}}\sup\;\abs{f(w)-f_{k}(w)}^2\le  2\cdot 15\cdot 2^4\tau_k^{10}+\frac{8}{\tau_{k+1}^2}\nonumber\\
	&\le \tau_k^{11}<\exp\bb{\tau_{k-1}^\alpha}\le \exp\bb{R^\alpha}.
\end{align}
That is, $M_f(R)\le e^{R^\alpha},$
implying that the order of $f$ is at most $\alpha$.

\begin{rmk}
	The sequence $\tau_{k+1}=\exp\bb{\tau_k^{\frac\alpha 7}}$ will give rise to an entire function of order $b$ for \underline{some} $b\in\bb{\frac\alpha7,\alpha}$. While it is possible to ensure the order is exactly $\alpha$, the additional details would have made the proof of our theorem, which is already highly technical, even longer. We therefore outline here the modifications required to obtain order exactly $\alpha$. \\
	For every $k$ we add to the construction of the subharmonic function, $u_k$, another disk, $B\bb{-\tau_k,\frac{\tau_k}4}$. In particular, a modified version of \eqref{eq:bnd_pst} holds on $B(-\tau_k,4)$, and $\left.u_k\right|_{\mathcal A_3}\equiv \abs z^\alpha$ for $\mathcal A_3:=\left\{\,z\in\mathbb{C}\colon  \frac{\tau_1}{4}<|z-\bb{-\tau_k}|<\frac{\tau_k}{3}\right\}$. We then modify the model map, $h_k$, assigning $-\exp\bb{\bb{\frac{\tau_k}{2}}^\alpha}$ on $B\bb{-\tau_k,\frac{10\tau_k}{36}}$. Estimating the integral, this  will not change it at all since $\abs{\nabla\chi_k}^2\sim\frac1{m\bb{\mathcal A_3}}$ and $\left.\abs{h_k(z)}\cdot e^{-\abs z^\alpha}\right|_{\mathcal A_3}\ll 1$. By increasing $\tau_1$ slightly, we can ensure the upper bound on the integral remains unchanged. In addition, following Proposition \ref{prop:method} combined with (the modified version of) \eqref{eq:bnd_pst} we see that $\abs{f_k(-\tau_k)-\bb{-\exp\bb{\bb{\frac{\tau_k}2}^\alpha}}}<\frac1{\tau_{k+1}}$. Finally, following property \ref{itm:consistency} of the sequence $\bset{f_k}$, we see that $\abs{f(-\tau_k)}\ge \frac12\exp\bb{\bb{\frac{\tau_k}2}^\alpha}$  for all $k$. In particular
	$$
	\limitsup R\infty\frac{\log\log\bb{M_f(R)}}{\log(R)}\ge\limit k \infty \frac{\log\log\abs{f(-\tau_k)}}{\log(\tau_k)}\ge\limit k \infty \frac{\log\log\bb{\frac12\exp\bb{\bb{\frac{\tau_k}2}^\alpha}}}{\log(\tau_k)}=\alpha.
	$$
	Combining this with \eqref{eq:growth-ratef} we see that $f$ is of order $\alpha$ (mean type).
\end{rmk}

\bigskip 

\paragraph*{Acknowledgements.}
We thank Lasse Rempe, Dave Sixsmith, Phil Rippon and Gwyneth Stallard for helpful discussions and for pointing us to relevant references. Part of this work was undertaken while the authors were affiliated with Northwestern University and the University of Manchester, respectively; we gratefully acknowledge support from these institutions, including funding for a research visit to Northwestern University. The first author acknowledges support from the Golda Meir Fellowship. The second author is a Serra Húnter fellow; this work was partially supported by the project PID2023-147252NB-I00 financed by MICIU/AEI MCIN/AEI/10.13039/501100011033, FEDER, EU, by the Spanish State Research Agency, through the Severo Ochoa and María de Maeztu Program for Centers and Units of Excellence in R\&D (CEX2020-001084-M).

\bibliographystyle{alpha}
\bibliography{biblio_boundariesWD_2}

\bigskip\bigskip\bigskip\bigskip\bigskip\bigskip

\noindent A.G.: Einstein Institute of Mathematics, Edmond J. Safra Campus, The Hebrew University of Jerusalem, Givat Ram. Jerusalem, 9190401, Israel
\newline{\tt https://orcid.org/0000-0002-6957-9431}
\newline{\tt adi.glucksam@mail.huji.ac.il}

\bigskip\bigskip
\noindent L.P.S. Dept. de Matemàtiques i Informàtica, Universitat de Barcelona, Catalonia, Spain,
\newline  Centre de Recerca Matemàtica, Bellaterra, Catalonia, Spain.
\newline {\tt https://orcid.org/0000-0003-4039-5556}
\newline{\tt lpardosimon@ub.edu}

\end{document}